\documentclass[a4paper, 12pt]{amsart}
\usepackage{amsmath, amsthm, amscd, amssymb, amsfonts, amsxtra, amssymb, latexsym}
\usepackage{enumerate}
\usepackage{verbatim}
\usepackage{float}
\usepackage{graphicx,epsfig,tikz}
\usepackage{pb-diagram}
\usepackage{scalerel}
\usepackage{tikz}
\usetikzlibrary{arrows.meta}

\usepackage{array}
\usepackage{hyperref}
\hypersetup{colorlinks = true,	allcolors  = blue}

\newcommand{\G}{\Gamma}
\newcommand{\Z}{\mathbb{Z}}

\newcommand{\C}{\mathbb{C}}
\newcommand{\N}{\mathbb{N}}
\newcommand{\ff}{\mathbb{F}}

\newcommand{\sk}{\smallskip}
\newcommand{\msk}{\medskip}

\newtheorem{thm}{Theorem}[section]
\newtheorem{prop}[thm]{Proposition}
\newtheorem{lem}[thm]{Lemma}
\newtheorem{coro}[thm]{Corollary}

\theoremstyle{definition}
\newtheorem{rem}[thm]{Remark}
\newtheorem{exam}[thm]{Example}

\theoremstyle{remark}

\usepackage{color}

\begin{document} \sloppy
	\numberwithin{equation}{section}
	\title[On $k$-UCG's over finite commutative rings]{On $k$-th unitary Cayley graphs over finite \\ commutative rings: 
		structure and decompositions}
	\author[R.A.\@ Podest\'a, D.E.\@ Videla]{Ricardo A.\@ Podest\'a, Denis E.\@ Videla}
	\dedicatory{\today}
	\thanks{Partially supported by CONICET and SECyT-UNC}
	
	\address{Ricardo A.\@ Podest\'a, FaMAF--CIEM (CONICET), Universidad Nacional de C\'ordoba, \newline
		Av.\@ Medina Allende 2144, Ciudad Universitaria, (5000) C\'ordoba, Argentina. 
		\newline {\it E-mail: podesta@famaf.unc.edu.ar}}
	\address{Denis E.\@ Videla, FaMAF--CIEM (CONICET), Universidad Nacional de C\'ordoba, \newline
		Av.\@ Medina Allende 2144, Ciudad Universitaria,  (5000) C\'ordoba, Argentina. 
		\newline {\it E-mail: devidela@famaf.unc.edu.ar}}

	\begin{abstract}
		Given $R$ a finite commutative ring with identity and $k\in \N$, we consider the \textit{$k$-th unitary Cayley graph} $G_R(k)=Cay(R,U_{R,k})$ with $U_{R,k}=\{x^k: x \in R^*\}$, and its symmetrized version $\mathcal{G}_R(k) = Cay(R,T_{R,k})$, with $T_{R,k}=U_{R,k} \cup (-U_{R,k})$. 
		If $R$ is a local ring with maximal ideal $\frak m$, we give the blow-up decompositions for the graphs: namely, we have 
		$G_R(k)= (G_{R/\frak m}(k))^{(|\frak m|)}$ and $\mathcal{G}_R(k)= (\mathcal{G}_{R/\frak m}(k))^{(|\frak m|)}$ for any $k$ such that $(k,|R|)=1$. 
		If the ring 
		$R$ has Artin decomposition $R=R_1 \times \cdots \times R_s$  in local rings $R_i$, we give the Kronecker product decompositions $G_R(k) = G_{R_1}(k) \otimes \cdots \otimes G_{R_s}(k)$ and $\mathcal{G}_R(k) = \mathcal{G}_{R_1}(k) \otimes \cdots \otimes \mathcal{G}_{R_s}(k)$. 
		In further $(k,|R|)=1$, these  
		decompositions can be given in terms of generalized Paley (GP) graphs over finite fields, that is  $G_{R_i}(k) = \G(k_i,q_i)$ and similarly for $\mathcal{G}_{R_i}(k)$, for $i=1,\ldots,s$. Also, the reduced graphs correspond to the graphs of the reduced rings, i.e.\@ 
		$\big(G_{R}(k)\big)_{red}\simeq G_{R_{red}}(k)$ and $\big(\mathcal{G}_{R}(k)\big)_{red} \simeq \mathcal{G}_{R_{red}}(k)$.
		By using these decompositions in terms of GP-graphs, we study some basic structural properties of the graphs such as directedness, bipartiteness and connectedness. 
	\end{abstract}

	\maketitle

	\section{Introduction: preliminaries and results} \label{sec: introduction}
	In this work, we focus on a generalization of GP-graphs over finite fields to finite rings, the so-called $k$-th power unitary Cayley graphs over rings, denoted $G_R(k)$, and their symmetrized versions $\mathcal{G}_R(k)$. We study their basic structure (directedness, bipartiteness, connectedness) and decompositions  (blow-ups and Kronecker products).

	\subsection{Preliminaries} \label{subsec: preliminares}
	Here we recall the basic definitions and results of Cayley graphs over rings and of products and blow-ups on arbitrary graphs.

	\subsubsection*{Cayley graphs over rings}
	Let $G$ be a finite commutative group and $S$ a subset of $G$ with $0\notin S$. The \textit{Cayley graph} 
	\begin{equation} \label{eq: XGS}
		X(G,S)
	\end{equation}  
	is the directed graph whose vertex set is $G$ and $v, w \in G$ form a directed edge (or arc) of $\Gamma$ from $v$ to $w$ if $w-v \in S$. Since $0\notin S$, $\Gamma$ has no loops.  Notice that if $S$ is symmetric, that is $-S=S$, then $X(G,S)$ is undirected (and conversely), and hence $X(G,S)$ is $|S|$-regular. In the case that $S$ is not symmetric, one can also consider the undirected graph 	
	$$ X(G,T) \qquad \text{ where } \qquad T = S\cup (-S) $$ 
	is the symmetrization of $S$. 
	The relation between the graphs $X(G,S)$ and $X(G, S \cup (-S))$, even for non-commutative groups $G$, can be found in Theorem 2.4 of \cite{PV19}.

	One interesting instance of this graph is when $G$ is a finite field. In particular,
	if we let $q=p^m$ with $p$ a prime number and $k$ a positive integer, 
	the \textit{generalized Paley graph} (\textit{GP-graph} for short) is the Cayley graph
	\begin{equation} \label{eq: Gkq def}
		\G(k,q) = X(\ff_{q},R_{k}) \qquad \text{with} \qquad R_{k} = \{ x^{k} : x \in \ff_{q}^*\}.
	\end{equation} 
	It is well-known that 
	\begin{equation} \label{eq: Gkq red}
		\G(k,q)=\G(k',q) \qquad \text{where} \qquad k'=\gcd(k,q-1),
	\end{equation} 
	so it is usual to assume that $k\mid q-1$.
	From now on, we will use $(a,b)$ in place of $\gcd(a,b)$ to denote the greatest common divisor of integers $a$ and $b$.
	
	Notice that, for $k\mid q-1$, the graph $\G(k,q)$ is an $n$-regular graph with $n=\tfrac{q-1}{k}$.
	The following properties of $\G(k,q)$ will be useful in the last sections:
	\begin{itemize}
		\item $\G(k,q)$ is undirected either if $q$ is even or if $k \mid \tfrac{q-1}2$ (i.e., $n$ is even) when $p$ is odd. 
		\medskip 
		
		\item $\G(k,q)$ is connected if $n$ is a primitive divisor of $q-1$.  
	\end{itemize}
	When $k=1$ we get the complete graph $\G(1,q)=K_q$ and when $k=2$ we get the classic Paley graph $\Gamma(2,q) = P(q)$. The graphs $\G(3,q)$ and $\G(4,q)$ are also of interest (see \cite{PV8}).
	
	GP-graphs have been extensively studied in the few past years. 
	Lim and Praeger studied their automorphism groups and characterized all GP-graphs which are Hamming graphs 
	(\cite{LP}). 
	Also, Pearce and Praeger characterized all GP-graphs which are Cartesian decomposable (\cite{PP}); some extensions to the directed case can be found in \cite{PV7}.
	The number of walks in GP-graphs is related with the number of solutions of 
	diagonal equations over finite fields (\cite{V}). 
	Under some mild restrictions, the spectrum of GP-graphs determines the weight distribution of their associated irreducible
	codes (\cite{PV4}). Furthermore, GP-graphs can also be seen as particular regular maps in Riemann surfaces 
	(\cite{J}, \cite{JW}).

	Another important special case of Cayley graphs is obtained when $G$ is a finite commutative ring with identity $R$ and $S$ is its group of units $R^*$. That is 
	\begin{equation} \label{GR} 
		G_R = X(R,R^*),
	\end{equation} 
	called the \textit{unitary Cayley graphs}.
	Unitary Cayley graphs were studied for instance in \cite{Ak+}, \cite{Minac}, \cite{Il}, \cite{Ki+}, \cite{LZ}, \cite{Nguyen} and \cite{PV5}, where some invariants (diameter, girth, chromatic numbers) and both structural aspects (automorphisms, connectedness, bipartiteness, primeness) and spectral features (spectrum, energy, equienergy, isospectrality) were treated. 
	Recently, Liu and Zhou \cite{LZ2} defined the \textit{quadratic unitary Cayley graphs} 
	\begin{equation} \label{GR quadratic} 
		\mathcal{G}_R = X(R,T_{R}),
	\end{equation} 
	where 
	$T_R=Q_{R} \cup (-Q_R)$ and $Q_{R}=\{x^2: x\in R^*\}$ 
	with $R$ a finite commutative ring with identity of odd characteristic. Under some conditions, they showed that in most of the cases these graphs are Kronecker products of classic Paley graphs $P(q)$ and the pseudograph $\mathring{K}_m$ resulting by attaching loops to all of the vertices of a complete graph (see Theorems 2.3 and 2.5 in \cite{LZ2}). This was shown by de Beaudrap in the case $R=\mathbb{Z}_n$ (see \cite{Be}).
	This decomposition allowed the author to determine the spectra, energy, and other interesting properties of $\mathcal{G}_R$.
	
	Here, we will extend the definition and results of Liu and Zhou to all powers $k\ge 2$.
	\goodbreak 
	
	\subsubsection*{Kronecker products of graphs and blow-ups} \label{subsec: blowups}
	The \textit{Kronecker product} of graphs $\G_1,\ldots, \G_s$, 
	$$\G=\G_1\otimes \cdots\otimes \G_s,$$ 
	is the graph with vertex set $V(\Gamma)=V(\G_1) \times \cdots \times V(\G_s)$ 
	where two vertices $v=(v_1,\ldots,v_s)$ and $w=(w_1,\ldots,w_s)$ 
	are neighbors in $\Gamma$ if and only if $v_i$ and $w_i$ are neighbors in the graph $\G_i$ for all $i=1,\ldots,s$.
	This product is associative and commutative, that is
	$$ \G_1 \otimes (\G_2 \otimes \G_3) = \G_1 \otimes \G_2 \otimes \G_3 = (\G_1 \otimes \G_2) \otimes \G_3 \qquad \text{and}
	\qquad \G_1 \otimes \G_2 = \G_2 \otimes \G_1 $$ 
	for any graphs $\G_1, \G_2, \G_3$.
	The Kronecker product extends naturally to pseudographs (directed graphs or graphs with loops). 
	In this case, notice that $\G_1 \otimes \G_2$ is loopless if one of $\G_1$, $\G_2$, is loopless, and $\G_1\otimes \G_2$ is directed if one of $\G_1$, $\G_2$, is directed.

	For any graph $\G$ and $t\in \N$, the \textit{(balanced) blow-up of order $m$} of $\G$, denoted 
	$\G^{(m)}$, 
	is the graph obtained by replacing each vertex $x$ of $\G$ by a set $V_x$ of $m$ independent vertices and every edge $\{x, y\}$ of $\G$ by a complete bipartite graph $K_{m,m}$ with parts $V_x$ and $V_y$. 
	Note that, we have the natural isomorphism 
	\begin{equation} \label{eq: blowup}
		\G \otimes \mathring{K}_m \simeq \G^{(m)},
	\end{equation}
	where $\mathring{K}_m$ is the graph obtained from the complete graph $K_m$ by attaching a loop to each of its vertices.

	\subsection{Two families of Cayley graphs}
	Here we introduce two new families of Cayley graphs over rings whose connection sets are defined in terms of the $k$-th powers of units. 
	
	\subsubsection*{Unitary Cayley graphs of $k$-th powers}
	Let $k\in\mathbb{N}$ and let $R$ be any commutative ring with identity. We consider the \textit{$k$-th unitary Cayley graph} or (\textit{$k$-UCG} for short) by 
	\begin{equation} \label{eq: GRk}
		G_R(k) = X(R,U_{R,k}) \qquad \text{where} \qquad U_{R,k} = \{x^k: x\in R^*\}
	\end{equation} 
	(previously defined in \cite{PV9}). We also define the \textit{symmetrized $k$-th unitary Cayley graph} as
	\begin{equation} \label{uGRk}
		\mathcal{G}_R(k) = X(R,T_{R,k}) \qquad \text{where} \qquad T_{R,k} = U_{R,k} \cup (-U_{R,k}),
	\end{equation}
	generalizing the definition given by Liu and Zhou for $k=1,2$ in \cite{LZ} and \cite{LZ2}.
	
	Although in 
	\eqref{eq: GRk} and \eqref{uGRk} the ring $R$ is general, from now on we will only consider finite commutative rings $R$ (except for an isolated example).
	Clearly, $G_R(1)=\mathcal{G}_R(1)$ is the unitary Cayley graph $G_R$ in \eqref{GR} and $\mathcal{G}_R(2)$ is the quadratic unitary Cayley graph $\mathcal{G}_R$ in \eqref{GR quadratic}. 
	Notice that $G_R$ is a spanning subgraph of $\mathcal{G}_R(k)$, i.e.\@ a subgraph having all the vertices. 
	We will sometimes refer to the graphs $G_R(k)$ and $\mathcal{G}_R(k)$ indistinctly as $k$-UCGs.
	For simplicity, when no confusion can arise, we will write $U_{k}$, $T_{k}$ instead of $U_{R,k}$, $T_{R,k}$, respectively.
	Also, we will refer to the graphs $G_R$ and $\mathcal{G}_R$ where we want to consider the whole family $G_R(k)$ and $\mathcal{G}_R(k)$ for all $k\in \N$.
	
	In general, $G_{R}(k)$ is directed and $\mathcal{G}_{R}(k)$ is undirected.
	Notice that $U_{R,k}$ is a symmetric set if and only if $-1\in U_{R,k}$, and in this case we clearly have that $G_{R}(k)=\mathcal{G}_{R}(k)$. 
	In particular, when $R$ is a finite field $\ff_q$ of cardinality $q=p^n$ then $G_R(k)$ is the generalized Paley graph 
	\begin{equation} \label{eq: GPkq}
		G_{\ff_q}(k) = \G(k,q) := X(\ff_q,\{x^k:x\in \ff_q^*\}).
	\end{equation}
	We will use the notation $G_{\ff_q}(k)$ in the general situation, while the notation $\G(k,q)$ will be used when it is understood that $k \mid q-1$.
	See \cite{PV3} (also \cite{PV4}, \cite{PV5}) for more information and results on the graphs $\G(k,q)$.

	\subsection{Outline and main results}
	Briefly, in Sections \ref{sec: reduction} and \ref{sec: rel entre GRk y mathcalGRk} we study reductions of the power $k$ and how to express the symmetrized (undirected) graphs $\mathcal{G}_R$ in terms of the graphs $G_R$. In Section \ref{sec: blowups} to \ref{sec: products of symmetrized GRks} we study two decompositions of the graphs $\mathcal{G}_R(k)$ and $G_R(k)$: blow-ups decompositions when $R$ is local, and Kronecker product decompositions when $R=R_1 \times \cdots \times R_s$ a product of local rings. Finally, in Sections \ref{sec: bipartiteness} and \ref{sec: connectivity} we study bipartiteness and connectedness of $G_R(k)$ and $\mathcal{G}_R(k)$ for any finite commutative ring $R$.

	We now explain in more detail what we do in this work. 
	In Sections \ref{sec: reduction}-\ref{sec: blowups}, $R$ will be a finite commutative local ring, while in Sections \ref{sec: product decompositions for GRk}-\ref{sec: connectivity}, a general finite commutative ring. 
	
	In the first two sections after the Introduction, we study the graphs $G_R(k)$ for $R$ local. In Section \ref{sec: reduction}, 
	we reduce the power $k$. More precisely, in Proposition \ref{prop: reduction GRk GRk'}, we show how to choose $k' \le k$ such that 		
	$$G_R(k)=G_R(k').$$  
	In Section \ref{sec: directedness}, we give necessary and sufficient conditions for $G_R(k)$ to be (un)directed when $R$ is of odd characteristic. 
	
	The aim of Section \ref{sec: rel entre GRk y mathcalGRk} is to express the graphs $\mathcal{G}_R$ in terms of the graphs $G_R$, when $R$ is of odd characteristic. First, in \S \ref{subsec:4.1} we obtain a 2-adic reduction of $G_R(k)$ by dividing $k$ by powers of 2 (see Proposition \ref{prop: Urk/2 = 2URk}). Then, in Theorem \ref{teo: Gk and Gk2} (our first main result) we show that
	$$ \mathcal{G}_R(k)=G_R(k) \qquad \text{or} \qquad \mathcal{G}_R(k) = G_R(\tfrac{k}{2^N}) $$ 
	if $G_R(k)$ is undirected or directed, respectively, where 
	$$ N=v_2(k)-v_2(q-1)+1$$ 
	and $v_2 $ denotes the 2-adic valuation of an integer.
	
	In Section 5 we study balanced blow-up decompositions for $G_R(k)$ and $\mathcal{G}_R(k)$ for $R$ a local ring, under the extra assumption $(k,|R|)=1$. 
	In Theorem \ref{teo: blowup GRk}, we recall the result 
	$$ G_R(k) \simeq {G_{\ff_q}(k)}^{(m)} $$ 
	obtained in \cite{PV9}. In Theorem \ref{teo: blowup mathcalGRk} (our second main result) we obtain the balanced blow-up decomposition 
	$$ \mathcal{G}_R(k) \simeq \mathcal{G}_{\ff_{q}}(k)^{(m)}.$$ 
	Furthermore, if $R$ has odd characteristic we have  
	$\mathcal{G}_R(k) \simeq {G_{\ff_q}(k)}^{(m)}$ if $v_{2}(k)< v_{2}(q-1)$ and 
	$\mathcal{G}_R(k) \simeq {G_{\ff_q}(\tfrac{k}{2^N})}^{(m)}$ if $v_{2}(k)\ge v_{2}(q-1)$, where $N$ is as above.
	
	In the next two sections we give Kronecker product decompositions for the graphs $G_R$ and $\mathcal{G}_R$ in the case $R$ is a finite commutative ring having Artin's decomposition 
	$$ R=R_1 \times \cdots \times R_s$$ 
	with $R_i$ local rings for $i=1,\ldots, s$. 
	We first show in Proposition \ref{prop: Kronecker GRk} that the Artin's decomposition of $R$ pass straight through a product of graphs, that is
	$$ G_{R_1 \times \cdots \times R_s}(k) \simeq G_{R_1}(k) \otimes \cdots \otimes G_{R_s}(k). $$ 
	However, the same product decomposition for the symmetrized graphs $\mathcal{G}_R$ is only valid with an extra assumption: in Theorem \ref{thm: conditions for mathcalGRk Kronecker} we show that 
	$$ \mathcal{G}_{R}(k) \simeq \mathcal{G}_{R_1}(k) \otimes \cdots \otimes \mathcal{G}_{R_s}(k) \quad \Leftrightarrow \quad -1\not \in U_{R_j,k} \quad \text{for at most one $j\in I_s$}$$
	where 	
	$$ I_s = \{1,\ldots,s\}. $$ 
	Furthermore, under the extra simple hypothesis $(k,|R|)=1$, which is not too restrictive, we can give product decompositions of the graphs $G_R$ and $\mathcal{G}_R$. These are the contents of Theorems \ref{teo: GRk tensor GRki's caso no local} and \ref{thm: product decomposition for mathcal GRk} (third and fourth main results) where we obtain that
	$$	G_{R}(k) \simeq \big( \bigotimes_{i=1}^s \G(k_i,q_i) \big) \otimes \mathring{K}_m \simeq G_{R_{red}}(k) \otimes 			\mathring{K}_m  \simeq (G_{R_{red}}(k))^{(m)},$$
	where $k_i=(k,q_i-1)$ for $i\in I_s$, $m$ is the size of the nilradical $\mathcal{N}_R$ of $R=R_1 \times \cdots \times R_s$ 
	and where $R_{red} = \ff_{q_1} \times \cdots \times \ff_{q_s}$ is the reduced ring of $R$; and also that 
	$$ 	\mathcal{G}_{R}(k) \simeq \big( \bigotimes_{i=1}^s \mathcal{G}_{\ff_{q_i}}(k) \big) \otimes \mathring{K}_m \simeq \mathcal{G}_{R_{red}}(k) \otimes \mathring{K}_m \simeq (\mathcal{G}_{R_{red}}(k))^{(m)}$$ 
	if and only if there is at most one $j \in I_s$ such that $-1 \notin U_{R_j,k}$. 
		Moreover, in Theorems \ref{thm: GR_red=G_Rred} and \ref{thm: symmetrized GR_red=G_Rred}, we show that if $(k,|R|)=1$, the reduced $k$-th unitary Cayley graphs associated to $R$ are the $k$-th unitary graphs associated to the reduced ring $R_{red}$, that is 
		$$ (G_R(k))_{red} \simeq G_{R_{red}}(k) \qquad \text{and} \qquad (\mathcal{G}_R(k))_{red} \simeq \mathcal{G}_{R_{red}}(k)	.$$ 
	
	In the last two sections, using the previous decompositions, we study conditions for the graphs $G_R$ and $\mathcal{G}_R$ to be bipartite and connected. The results in these sections are more laborious to obtain.  
	
	In Section \ref{sec: bipartiteness}, using the Kronecker product decompositions above, in Theorem \ref{teo: GR bip} (our fifth main result) we give conditions for $G_R(k)$ and $\mathcal{G}_R(k)$ to be bipartite.  
	In fact, the theorem shows that these graphs are almost always non-bipartite. More precisely, $G_R(k)$ is bipartite if and only if there is at least one index $i\in I_s$ such that  $G_{\ff_{q_i}}(k)$ equals 
	\begin{equation} \label{eq: bip condition intro}
		\G(2^{t_i}-1,2^{t_i}) \simeq K_2 \sqcup \cdots \sqcup K_2 \quad (\text{$2^{t_i-1}$ copies}),	
	\end{equation} 
	where $q_i=2^{t_i}$ and $(k,q_i-1)=2^{t_i}-1$.
	If further, there is at most one index $j \in I_s$ such that $-1\not \in U_{R_j,k}$ then, $\mathcal{G}_R(k)$ is bipartite if and only if  there is some index $i \in I_s$, such that 
	$\mathcal{G}_{\ff_{q_i}}(k)$ equals \eqref{eq: bip condition intro}.
	
	To conclude, in Section \ref{sec: connectivity}, after a detailed study of connectedness of Kronecker products of arbitrary graphs (considering separately the cases undirected, directed and mixed) and the Kronecker product decompositions above, we can give results on the connectedness of the graphs $G_R$ and $\mathcal{G}_R$. 
	In Theorem \ref{teo: GR connected} (our last main result) we show that under the hypothesis $(k,|R|)=1$ and some extra mild arithmetic conditions,  
	$G_{R}(k)$ is (strongly) connected if and only if there is no two indices $i,j\in I_s$ such that 
	$$G_{R_{i}/\frak m_i}(k) \simeq \G(p-1,p) \simeq G_{R_{j}/\frak m_j}(k)$$ 
	for some prime $p$. 
	Finally, in Theorem \ref{teo: GR sim conn} 
	we relate the connectedness of $G_R$ with the one of $\mathcal{G}_R$. 
	Under the same hypothesis of Theorem \ref{teo: GR connected}, we get that 
	$$ \mathcal{G}_{R}(k) \text{ is connected} \qquad \Leftrightarrow \qquad G_{R}(k) \text{ is (strongly) connected} $$
	if there is at most one index $j \in I_s$ such that $-1 \notin U_{R_j,k}$.
	
	As a final comment, we want to point out the fact that although the symmetrized graphs $\mathcal{G}_R$ are always undirected, all the results for $\mathcal{G}_R$ are more involved than the corresponding ones for $G_R$, needing more hypotheses --as a general rule-- to get a similar result.

	\section{Reduction of $k$ in $G_R(k)$, for $R$ finite local}  \label{sec: reduction}
	In this section, we show that for the graphs $G_R(k)$ with 
	$$ \text{$R$ a \textsl{finite commutative local ring}}$$ 
	one can always reduce the power $k$, that is, there exist a $k'\le k$ such that $G_R(k) = G_R(k')$. 
	The details are the content of the next proposition.
	
	\begin{prop} \label{prop: reduction GRk GRk'}
		Let $(R, \frak{m})$ be a finite commutative local ring with residue field $R/\frak{m}~\simeq~\ff_{q}$ with
		$q$ a power of a prime $p$ and let 
		$$k=p^{h}\ell \in \N$$ 
		with $h\ge 0$ and $\ell \ge 1$ such that $(\ell,p)=1$.
		If $|\frak{m}|=p^{t}$, then 
		\begin{equation} \label{eq: GRk=GRk'}
			G_R(k) = G_R(k') \qquad \text{with} \qquad k'=p^{\min\{t,h\}}\ell' \le k
		\end{equation}	
		where $\ell' = (\ell,q-1)$. 
	\end{prop}

	\begin{proof}
		It is enough to show that the connection sets of the graphs are equal, that is 
		$$ U_{R,k} = U_{R,k'}. $$ 
		
		Let $k=p^{h}\ell$, with $(p,\ell)=1$.
		Since $(\ell,|R|)=1$, as a consequence of Theorem 2.3 in \cite{PV9}, 
		if we put $\ell'=(\ell,q-1)$ then we have that 
		\begin{equation}\label{eq: Ul Ul'}
			U_{R,\ell} = U_{R,\ell'}. 
		\end{equation}
		
		Consider the map $F_p$ on $R$ defined by 
		$F_{p}(x) = x^{p}$.
		Thus, it is clear that 
		$$ (F_{p})^{h}(x) = x^{p^{h}}. $$ 
		Also, for any $u\in R$ we have that $u^{k}=(u^{\ell})^{p^{h}}$. 
		Therefore, by using \eqref{eq: Ul Ul'} we obtain
		\begin{equation} \label{eq: URk=URphl}
			U_{R,k} = U_{R,p^{h} \ell} = (F_{p})^{h} \big( U_{R,\ell}\big) = (F_{p})^{h} \big( U_{R,\ell'}\big) = U_{R,p^{h}\ell'}.
		\end{equation}
		
		Now, by \eqref{eq: URk=URphl} the statement in the proposition is obvious when $h<t$.
		Thus, let us assume that $h\ge t$. Hence, it is enough to show that $G_{R}(k)=G_{R}(p^{t} \ell')$ in this case. 
		
		Let $s: R/\frak{m} \rightarrow R$ be a section of the quotient map $\pi: R \rightarrow R/\frak{m}$, such that 
		$$ s([-1]) = -1 \qquad \text{and} \qquad s([1])=1. $$ 
		Since $\pi \circ s =Id$, the short exact sequence 
		$$
		1 \rightarrow (R/\frak{m})^* \stackrel{s}{\rightarrow} R^* \stackrel{r}{\rightarrow} 1+\frak{m} \rightarrow 1 
		$$
		splits, where the morphism $r$ is given by $r(u) = s([u])^{-1}u$. 
		In particular, we  have that $s([u])^{-1} u \in 1+ \frak{m}$.
		Thus, the splitting short exact sequence induces the isomorphism 
		\begin{equation} \label{eq: Psi isomorfismo}
			\Psi: R^* \rightarrow (R/\frak{m})^* \times (1+\frak{m}), \qquad \Psi(u)=([u],s([u])^{-1}u).
		\end{equation}	
		
		Notice that $1+\frak{m}$ is a commutative ring of size $p^{t}$. 
		By taking into account that $u^{k}=(u^{p^h})^{\ell}$, if $h\ge t$, then $c^{k}=1$ for any $c\in 1+\frak{m}$ due to the Lagrange Theorem. In particular, we have that  
		$$ (s([u])^{-1}u)^{k} = 1 $$ 
		for any $u \in R$. 
		Then, the image of $U_{R,k}$ under the isomorphism $\Psi$ is $U_{\ff_q,k}\times\{1\}$, and in the same way we get that $\Psi(U_{R,p^{t}\ell}) = U_{\ff_{q},p^{t}\ell}\times \{1\}$. 
		
		Next, since $\Psi_{p}(x)=x^p$ is a field automorphism on $\ff_q$, by \eqref{eq: URk=URphl} we have that 	
		$$ U_{\ff_q,k} = U_{\ff_q,\ell} = U_{\ff_{q},p^{t}\ell} \qquad \text{and} \qquad U_{\ff_q,\ell}=U_{\ff_q,\ell'}.$$  
		Hence $U_{\ff_q,k} = U_{\ff_q,\ell'}$, showing that $G_{R}(k)=G_{R}(p^{t} \ell')$.
		
		Putting the cases $h<t$ and $h\ge t$ together, we have shown that $G_R(k) = G_R(k')$ with $k'=p^{h'} \ell'$,
		where $h'=\min\{t,h\}$, as asserted. The fact that $k'\le k$ is clear by the definition of $k'$, and we are done.	
	\end{proof}

	\begin{rem}
		In the proposition, we have that $R=\ff_q$ is a finite field if and only if $|\frak{m}|=1$, that is, if and only if $t=0$. In this case, we re-obtain the known reduction in \eqref{eq: Gkq red} for GP-graphs $\G(k,q)$.
	\end{rem}

	It is interesting to study this reduction in more detail for small values of $p$, and we leave the details of this to the reader. 
	However, we consider a case where the reduction can be done easily in general, i.e., for 
	commutative local rings having order $p^2$.

	\begin{coro}\label{coro: red res p} 
		Let $R$ be a commutative local ring of order $p^2$ with $p$ prime, i.e.\@ $R=\Z_{p^2}$ or $R=\ff_p[x]/(x^2)$.
		Then, given $k\in \N$ we have that 
		$$ G_R(k)=G_R(k') \quad \text{for some} \quad k'\le p(p-1). $$ 
		More precisely, if $\ell \in \N$ is such that $(\ell,p)=1$ and we put $\ell'=(\ell,p-1)$, then we have 
		$$G_R(\ell) =G_R(\ell') \qquad \text{and} \qquad G_R(p^h \ell) =G_R(p\ell') \quad \text{for any } h\in \N.$$
		In particular, for any $k$ coprime with $p$ and $p-1$ we have that $G_R(k) = G_R(1)$. 
	\end{coro}
	
	\begin{proof}
		Since $R$ has order $p^2$, the residue field $R/\frak m$ of $R$ is $\ff_p$. Thus, we have that $t=1$ in \eqref{eq: GRk=GRk'}. 
		If we write $k=p^h \ell$ with $(\ell,p)=1$, then we have that 
		$$ k'= p^{\min\{1,h\}} \ell'$$ 
		with $\ell'=(\ell,p-1) \le p-1$. 
		Hence, we get that $k'\le p(p-1)$. 
		
		For the remaining assertions, notice that if $(k,p)=1$ then $v_p(k)=0$ (where $v_p$ is the $p$-adic valuation of $k$) and hence $\ell'=1$, and so the result follows from Proposition \ref{prop: reduction GRk GRk'}.
	\end{proof}
	
	We now show how to use the reduction of the proposition.
	\begin{exam}
		Take $p=3$ and let $R_9$ be a finite commutative local ring of order $9$ (i.e., $\Z_9$ or $\Z_3[x]/(x^2)$). 
		In this case $t=1$ in \eqref{eq: GRk=GRk'}.
		If $k=3^h \ell$ with $(\ell,3)=1$, then $\ell'=(\ell,3-1) \in \{1,2\}$.
		So, $k'=3^{\min\{1,h\}} \ell'$ 
		can only take the values 
		$$ k'=1,2,3,6;$$ 
		namely $k'=1,2,$ if $h=0$ and $k'=3,6,$ if $h\ge 1$.
		
		Now, since $\ell$ is coprime with $3$, we have that $\ell = 3r+j$ with $j=1,2$ and $r\ge 0$. Thus, we get 
		$$ G_{R_9}(3^h(3r+j)) = \begin{cases}
			G_{R_9}(1) \qquad \quad \text{if $h=0$, $3r+j$ odd}, \\[1mm]
			G_{R_9}(2) \qquad \quad \text{if $h=0$, $3r+j$ even}, \\[1mm]
			G_{R_9}(3) \qquad \quad \text{if $h\ge 1$, $3r+j$ odd}, \\[1mm]
			G_{R_9}(6) \qquad \quad \text{if $h\ge 1$, $3r+j$ even}, \\
		\end{cases} $$
		for any $r,h \in \N_0$.

		Moreover, these are known graphs. Indeed, any $G_{R_9}(k)$ is one of the graphs $K_{3 \times 3}$, $\vec{K}_{3 \times 3}$, $C_{9}$ or $\vec{C}_9$ (see the figures below).
		In fact, take $R_9=\Z_9$ (the graphs $G_{R_9}(k)$ for $R_9=\Z_9$ and $R_9=\Z_3[x]/(x^2)$ are isomorphic, but it is easier to work with $\Z_9$), hence $\Z_9^* = \{1,2,4,5,7,8\}$ and thus 
		$$ U_{\Z_9,1} = \{1,2,4,5,7,8\}, \quad U_{\Z_9,2} = \{1,4,7\}, \quad 
		U_{\Z_9,3} = \{1,8\}, \quad \text{and} \quad U_{\Z_9,1} = \{1\}.$$ 
		This implies that the graphs $G_{\Z_9}(k)$ are $\frac{6}k$-regular with $k=1,2,3,6$. 
		One can show (or see Example \ref{exam: mulipatito}) 
		that 
		$$ G_{R_9}(1) = K_{3 \times 3} = K_{3,3,3} $$ 
		is the complete tripartite graph with $3$ independent sets of size $3$.
		In a similar way, it can be seen that $G_{R_9}(2) = \vec{K}_{3 \times 3}$ is a directed version of $K_{3 \times 3}$. In fact, $\vec{K}_{3 \times 3}$ is obtained from $K_{3 \times 3}$ by orienting all its edges in the same direction.
		Also, since $U_{{R_9},6}=\{1\}$ and $U_{{R_9},3}=\{1,-1\}$, it is clear that $G_{R_9}(3)=C_{9}$ is the $9$-cycle graph and $G_{R_9}(6)=\vec{C}_9$ is the directed $9$-cycle graph. 
		
		\begin{figure}[H]
			\begin{minipage}{0.45\textwidth}
				\centering
				
				\begin{tikzpicture}[scale=0.75, auto, 
					thick, main node/.style={fill=black, circle, inner sep=2pt}, 
					label distance=1mm] 
					
					\foreach \i/\name/\pos in {
						0/$0$/right, 
						1/$3$/above right, 
						2/$6$/above, 
						3/$1$/above, 
						4/$4$/above left, 
						5/$7$/left, 
						6/$2$/below left, 
						7/$5$/below , 
						8/$8$/below right} 
					{
						\node[main node] (n\i) at ({360/9 * \i}:3cm) [label=\pos:\name] {}; 
					}
					
					
					
					\path[thick, draw=black] (n0) edge (n3); 
					\path[thick, draw=black] (n0) edge (n4); 
					\path[thick, draw=black] (n0) edge (n5); 
					\path[thick, draw=black] (n0) edge (n6);
					\path[thick, draw=black] (n0) edge (n7);
					\path[thick, draw=black] (n0) edge (n8);
					
					\path[thick, draw=black] (n1) edge (n3); 
					\path[thick, draw=black] (n1) edge (n4); 
					\path[thick, draw=black] (n1) edge (n5); 
					\path[thick, draw=black] (n1) edge (n6);
					\path[thick, draw=black] (n1) edge (n7);
					\path[thick, draw=black] (n1) edge (n8);
					
					\path[thick, draw=black] (n2) edge (n3); 
					\path[thick, draw=black] (n2) edge (n4); 
					\path[thick, draw=black] (n2) edge (n5); 
					\path[thick, draw=black] (n2) edge (n6);
					\path[thick, draw=black] (n2) edge (n7);
					\path[thick, draw=black] (n2) edge (n8);
					
					\path[thick, draw=black] (n3) edge (n0); 
					\path[thick, draw=black] (n3) edge (n1); 
					\path[thick, draw=black] (n3) edge (n2); 
					\path[thick, draw=black] (n3) edge (n6); 
					\path[thick, draw=black] (n3) edge (n7);
					\path[thick, draw=black] (n3) edge (n8);

					\path[thick, draw=black] (n4) edge (n0); 
					\path[thick, draw=black] (n4) edge (n1); 
					\path[thick, draw=black] (n4) edge (n2); 
					\path[thick, draw=black] (n4) edge (n6); 
					\path[thick, draw=black] (n4) edge (n7);
					\path[thick, draw=black] (n4) edge (n8);
					
					\path[thick, draw=black] (n5) edge (n0); 
					\path[thick, draw=black] (n5) edge (n1); 
					\path[thick, draw=black] (n5) edge (n2); 
					\path[thick, draw=black] (n5) edge (n6); 
					\path[thick, draw=black] (n5) edge (n7);
					\path[thick, draw=black] (n5) edge (n8);
					
					\path[thick, draw=black] (n6) edge (n0); 
					\path[thick, draw=black] (n6) edge (n1); 
					\path[thick, draw=black] (n6) edge (n2); 
					\path[thick, draw=black] (n6) edge (n3); 
					\path[thick, draw=black] (n6) edge (n4);
					\path[thick, draw=black] (n6) edge (n5);
					
					\path[thick, draw=black] (n7) edge (n0); 
					\path[thick, draw=black] (n7) edge (n1); 
					\path[thick, draw=black] (n7) edge (n2); 
					\path[thick, draw=black] (n7) edge (n3); 
					\path[thick, draw=black] (n7) edge (n4);
					\path[thick, draw=black] (n7) edge (n5);
					
					\path[thick, draw=black] (n8) edge (n0); 
					\path[thick, draw=black] (n8) edge (n1); 
					\path[thick, draw=black] (n8) edge (n2); 
					\path[thick, draw=black] (n8) edge (n3); 
					\path[thick, draw=black] (n8) edge (n4);
					\path[thick, draw=black] (n8) edge (n5);
					
				\end{tikzpicture}
			\end{minipage}
			\hspace{.5cm} 
			\begin{minipage}{0.4\textwidth}
				\centering
				
				\begin{tikzpicture}[scale=0.75, auto, 
					thick, main node/.style={fill=black, circle, inner sep=1pt}, 
					label distance=1mm] 
					
					\foreach \i/\name/\pos in {
						0/$0$/right, 
						1/$3$/above right, 
						2/$6$/above, 
						3/$1$/above, 
						4/$4$/above left, 
						5/$7$/left, 
						6/$2$/below left, 
						7/$5$/below , 
						8/$8$/below right} 
					{
						\node[main node] (n\i) at ({360/9 * \i}:3cm) [label=\pos:\name] {}; 
					}
					
					
					
					\path[thick, draw=black, -{Stealth}] (n0) edge (n3); 
					\path[thick, draw=black, -{Stealth}] (n0) edge (n4); 
					\path[thick, draw=black, -{Stealth}] (n0) edge (n5);

					\path[thick, draw=black, -{Stealth}] (n1) edge (n3); 
					\path[thick, draw=black, -{Stealth}] (n1) edge (n4); 
					\path[thick, draw=black, -{Stealth}] (n1) edge (n5); 
					
					\path[thick, draw=black, -{Stealth}] (n2) edge (n3); 
					\path[thick, draw=black, -{Stealth}] (n2) edge (n4); 
					\path[thick, draw=black, -{Stealth}] (n2) edge (n5); 
					
					\path[thick, draw=black, -{Stealth}] (n3) edge (n6); 
					\path[thick, draw=black, -{Stealth}] (n3) edge (n7);
					\path[thick, draw=black, -{Stealth}] (n3) edge (n8);

					\path[thick, draw=black, -{Stealth}] (n4) edge (n6); 
					\path[thick, draw=black, -{Stealth}] (n4) edge (n7);
					\path[thick, draw=black, -{Stealth}] (n4) edge (n8);
					
					\path[thick, draw=black, -{Stealth}] (n5) edge (n6); 
					\path[thick, draw=black, -{Stealth}] (n5) edge (n7);
					\path[thick, draw=black, -{Stealth}] (n5) edge (n8);
					
					\path[thick, draw=black, -{Stealth}] (n6) edge (n0); 
					\path[thick, draw=black, -{Stealth}] (n6) edge (n1); 
					\path[thick, draw=black, -{Stealth}] (n6) edge (n2); 
					
					\path[thick, draw=black, -{Stealth}] (n7) edge (n0); 
					\path[thick, draw=black, -{Stealth}] (n7) edge (n1); 
					\path[thick, draw=black, -{Stealth}] (n7) edge (n2); 
					
					\path[thick, draw=black, -{Stealth}] (n8) edge (n0); 
					\path[thick, draw=black, -{Stealth}] (n8) edge (n1); 
					\path[thick, draw=black, -{Stealth}] (n8) edge (n2); 
				\end{tikzpicture}
			\end{minipage}
			\caption{The graphs $G_{R_9}(1) \simeq K_{3,3,3}$ and $G_{R_9}(2) \simeq \vec{K}_{3,3,3}$.}
		\end{figure}
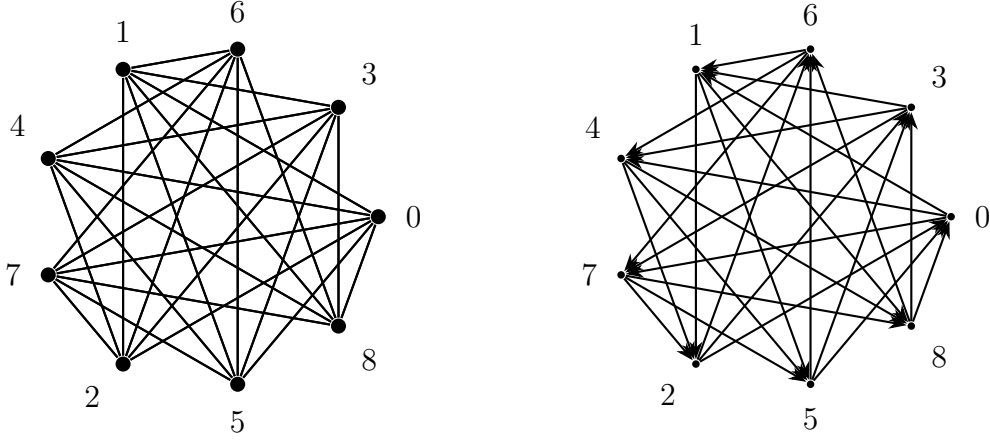
		
		\begin{figure}[H]
			\begin{minipage}{0.45\textwidth}
				\centering
				
				\begin{tikzpicture}[scale=0.75, auto, 
					thick, main node/.style={fill=black, circle, inner sep=2pt}, 
					label distance=1mm] 
					
					\foreach \i/\name/\pos in {
						0/$0$/right, 
						1/$1$/above right, 
						2/$2$/above, 
						3/$3$/above, 
						4/$4$/above left, 
						5/$5$/left, 
						6/$6$/below left, 
						7/$7$/below , 
						8/$8$/below right} 
					{
						\node[main node] (n\i) at ({360/9 * \i}:3cm) [label=\pos:\name] {}; 
					}
					

					\path[thick, draw=black] (n0) edge (n1); 
					\path[thick, draw=black] (n0) edge (n8); 
					
					\path[thick, draw=black] (n1) edge (n2); 
					\path[thick, draw=black] (n1) edge (n0); 
					
					\path[thick, draw=black] (n2) edge (n3); 
					\path[thick, draw=black] (n2) edge (n1); 
					
					\path[thick, draw=black] (n3) edge (n4); 
					\path[thick, draw=black] (n3) edge (n2); 
					
					\path[thick, draw=black] (n4) edge (n5); 
					\path[thick, draw=black] (n4) edge (n3); 
					
					\path[thick, draw=black] (n5) edge (n6); 
					\path[thick, draw=black] (n5) edge (n4); 
					
					\path[thick, draw=black] (n6) edge (n7); 
					\path[thick, draw=black] (n6) edge (n5); 
					
					\path[thick, draw=black] (n7) edge (n8); 
					\path[thick, draw=black] (n7) edge (n6); 
					
					\path[thick, draw=black] (n8) edge (n0); 
					\path[thick, draw=black] (n8) edge (n7); 
					
				\end{tikzpicture}
			\end{minipage}
			\hspace{.5cm} 
			\begin{minipage}{0.45\textwidth}
				\centering
				
				\begin{tikzpicture}[scale=0.75, auto, 
					thick, main node/.style={fill=black, circle, inner sep=2pt}, 
					label distance=1mm] 
					
					\foreach \i/\name/\pos in {
						0/$0$/right, 
						1/$1$/above right, 
						2/$2$/above, 
						3/$3$/above, 
						4/$4$/above left, 
						5/$5$/left, 
						6/$6$/below left, 
						7/$7$/below , 
						8/$8$/below right} 
					{
						\node[main node] (n\i) at ({360/9 * \i}:3cm) [label=\pos:\name] {}; 
					}
					
					
					\path[thick, draw=black, -{Stealth}] (n0) edge (n1); 
					
					\path[thick, draw=black, -{Stealth}] (n1) edge (n2); 
					
					\path[thick, draw=black, -{Stealth}] (n2) edge (n3); 
					
					\path[thick, draw=black, -{Stealth}] (n3) edge (n4); 
					
					\path[thick, draw=black, -{Stealth}] (n4) edge (n5); 
					
					\path[thick, draw=black, -{Stealth}] (n5) edge (n6); 
					
					\path[thick, draw=black, -{Stealth}] (n6) edge (n7); 
					
					\path[thick, draw=black, -{Stealth}] (n7) edge (n8); 
					
					\path[thick, draw=black, -{Stealth}] (n8) edge (n0); 
					
				\end{tikzpicture}
			\end{minipage}
			\caption{The graphs $G_{R_9}(3) \simeq C_9$ and $G_{R_9}(6) \simeq \vec{C}_9$.}
		\end{figure}
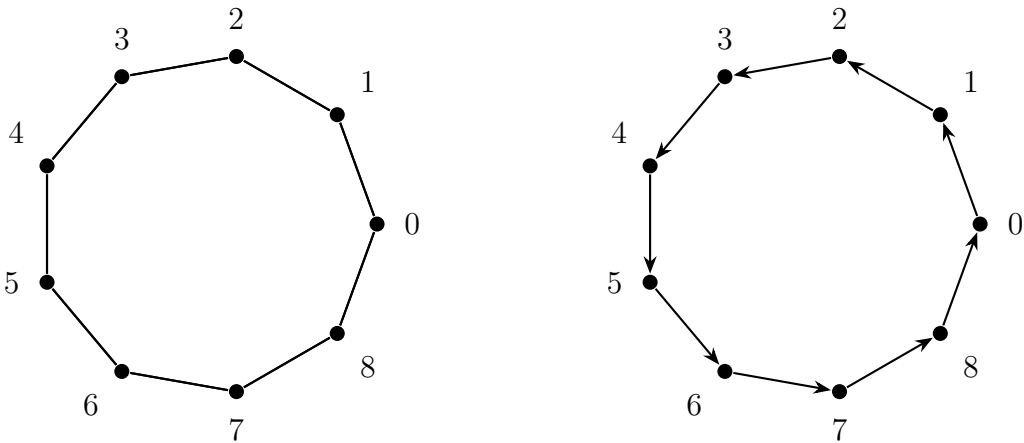
		The independent sets for $G_{\Z_9}(1) \simeq K_{3,3,3}$ and $G_{\Z_9}(2) \simeq \vec{K}_{3,3,3}$ are $V_1=\{0,3,6\}$, $V_2=\{1,4,7\}= V_1+1$ and $V_3=\{2,5,8\}=V_1+2$.
		\hfill $\diamond$
	\end{exam}

	To finish the section, we note that Proposition \ref{prop: reduction GRk GRk'} does not hold in general for an infinite local ring, as the next example shows.
	
	\begin{exam}
		Consider the ring of formal series with coefficients in $\ff_3$, that is 
		$$ R = \ff_{3}[\![x]\!] = \big\{ \sum\nolimits_{i=0}^{\infty} a_i x^i : a_i \in\ff_3 \big\} .$$ 
		Then, $R$ is an infinite local ring with maximal ideal $\frak m = \{ \sum_{i=1}^{\infty}a_i x^i : a_i \in\ff_3\}$ and residue field $R/\frak{m} = \ff_3$,  and hence $t=1$ in \eqref{eq: GRk=GRk'}.
		
		Take any even number $k>2$ which is not divisible by 3 (for instance $k=2^r$). Hence $h=0$ in \eqref{eq: GRk=GRk'} and we have 
		$ k'= 3^{\min\{1,0\}} (k,3-1)= (k,2)=2$.
		However, 
		$$U_{R,2} \ne U_{R,k}$$ 
		since there is no polynomial of degree $2$ in the set $U_{R,k}$, but the polynomial $(x+1)^{2}\in U_{R,2}$. 
		Thus, we have that $G_R(2) \ne G_R(k)$, although $k'=2$.
		\hfill $\diamond$
	\end{exam}

	\section{Directedness of $G_R(k)$, $R$ local} \label{sec: directedness}
	In this section, we show that for 
	$$ (R,\frak{m}) \text{\textsl{ a commutative local ring of odd characteristic}}, $$ 
	the graph $G_R(k)$ is (un)directed if and only if $G_{R/\frak{m}}(k)$ is (un)directed.
	That is, to study the directedness of the graphs $G_R$ one can change the local ring by the residue field, which is smaller. This is another kind of reduction that is useful.

	In what follows we will need the \textit{$2$-adic valuation} of an integer $k$, denoted by $v_2(k)$, which is the maximum non-negative integer $t$ such that $2^{t}\mid k$ and $2^{t+1}\nmid k$.  That is, 
	$$v_2(2^t m) = t \qquad (\text{$m$ odd}). $$
	
	We now give necessary and sufficient conditions for $-1$ to be a $k$-th power of $R$ in terms of 2-adic valuations of $k$ and $q-1$.
	
	\begin{lem} \label{lem: condition for -1 in URk, R local}
		If $(R, \frak{m})$ is a commutative local ring of odd characteristic with $R / \frak{m} \simeq \ff_{q}$, then we have: 
		
		\begin{enumerate}[$(a)$]
			\item $-1\in U_{R,k}$ if and only if $-1\in U_{R/\frak{m}, k}$. \smallskip
			
			\item $-1\in U_{R,k}$ if and only if $v_{2}(k)< v_{2}(q-1)$. \smallskip
			
			\item If $-1 \not \in U_{R,k}$, then $k \in 2\Z$  and $U_{R,k}\cap (-U_{R,k})=\varnothing$.
		\end{enumerate}
	\end{lem}

	\begin{proof}
		($a$) 
		Recall that we have the isomorphism $\Psi: R^* \rightarrow (R/\frak{m})^* \times (1+\frak{m})$ defined by $\Psi(u)=([u],(s([u]))^{-1}u)$, where $s: R/\frak{m} \rightarrow R$ is a section of the quotient map $\pi: R \rightarrow R/\frak{m} $ such that $s([-1])=-1$ and $s([1])=1$ (see \eqref{eq: Psi isomorfismo} in Proposition~\ref{prop: reduction GRk GRk'}). 
		
		Notice that $\Psi(-1)=(-[1],1)$ and, if $\Psi(a)=(b,c)$ then $\Psi(a^{k})=(b^{k},c^{k})=(b,c)^{k}$, which implies that 
		$$
		-1\in U_{R,k} \quad \Rightarrow \quad -1\in U_{R/\frak{m}, k}.
		$$ 
		
		On the other hand, suppose that there exists $\alpha \in (R/\frak{m})^*$ such that $\alpha^{k}=[-1]$. 
		Now, let $u=s(\alpha)$, since $s$ is a section we have that $s(\alpha \cdot \beta )= s(\alpha) s(\beta)$ which implies that  
		$$	
		u^{k}=s(\alpha^k) = s([-1]) = -1 \qquad \text{and} \qquad 
		(s(\alpha)^{-1})^{k}=s(\alpha^{k})^{-1}=(-1)^{-1}=-1. 
		$$
		Thus, 
		$$
		(\alpha,s(\alpha)^{-1}u)^{k}=(\alpha^{k}, (s(\alpha)^{-1})^{k} u^{k})=([-1],1).
		$$
		Hence, $([-1],1)$ is a $k$-th power in $(R/\frak{m})^* \times (1+\frak{m})$ and so $-1 \in U_{R,k}$ and therefore 
		\begin{equation} \label{eq: -1's}
			-1\in U_{R,k} \quad \Leftrightarrow \quad -1\in U_{R/\frak{m}, k}.
		\end{equation}
		
		\noindent ($b$) 
		Recall that in general, for an odd $q$ prime power we have that $U_{\mathbb{F}_{q},t}=U_{\mathbb{F}_{q},t'}$ where $t'=(t,q-1)$. 
		Moreover, $-1\in U_{\mathbb{F}_{q},t'}$ if and only if $t'\mid \frac{q-1}{2}$.  
		Hence, these two statements imply that $-1 \in U_{(R/\frak{m}), k}$ if and only if $v_{2}(k)\le v_{2}(\frac{q-1}{2})$. 
		Since 
		$$ v_{2}(\tfrac{q-1}{2}) = v_{2}(q-1)-1.$$
		The above statement is equivalent to 
		$-1 \in U_{(R/\frak{m}), k}$ if and only if $v_{2}(k)< v_{2}(q-1)$. Therefore, equation \eqref{eq: -1's} implies that
		$-1\in U_{R,k}$ if and only if $v_{2}(k)<v_{2}(q-1)$, as asserted.
		
		\smallskip 
		
		\noindent ($c$) 
		Assume that $-1 \not \in U_{R,k}$. Thus, $k$ must be even since $-1=(-1)^{k}\in U_{R,k}$ if $k$ is odd. 
		Now, suppose that $U_{R,k}\cap (-U_{R,k}) \ne \varnothing$. That is, there exists $r\in R^*$ such that
		$$ x^{k} = r = -y^{k} $$
		for some $x,y \in R^*$. 
		This implies that $-1=(xy^{-1})^{k} \in U_{R,k}$, which is a clear contradiction, and so $U_{R,k} \cap (-U_{R,k}) = \varnothing$, as we wanted to show.
	\end{proof}
	
	As a direct consequence of the previous result, we get that the UCGs over $R$ and its residue field $R/\frak m$ have the same directedness behavior. 
	
	\begin{prop}
		Let $(R, \frak{m})$ be a commutative local ring of odd characteristic with $R / \frak{m} \simeq \ff_{q}$. 
		Then, $G_{R}(k)$ is (un)directed if and only if $G_{R/\frak{m}}(k)$ is (un)directed. 
	\end{prop}
	
	\begin{proof}
		It is straightforward from $(a)$ of Lemma \ref{lem: condition for -1 in URk, R local}.
	\end{proof}
	
	We now illustrate the proposition with a general example.
	
	\begin{exam}
		Consider the finite local ring $R= \Z_{p^t}$ with $p$ an odd prime and $t\in \N$. Then, $R/\frak{m}=\ff_p$. 
		In this way, $G_{\Z_{p^t}}(k)$ is (un)directed if and only if $G_{\ff_p}(k)$ is (un)directed. 
		If $k\mid p-1$, then $G_{\ff_p}(k) = \Gamma(k,p)$; and it is known that this graph is undirected if and only if $k\mid \frac{p-1}{2}$. 
		In particular, we have that 
		$$ G_{\Z_{p^t}}(\tfrac{p-1}2) \text{ is undirected} \qquad \text{and} \qquad G_{\Z_{p^t}}(p-1) \text{ is directed},$$ 
		since we have 
		$G_{\ff_p}(\frac{p-1}2) = \G(\frac{p-1}2,p) = C_p$ and $G_{\ff_p}(p-1) = \G(p-1,p) = \vec{C}_p$, by the proposition. 
		\hfill $\diamond$
	\end{exam}

	However, the previous proposition is not true in even characteristic. 
	
	\begin{exam} \label{exam: not true -1 GRk}
		Consider the local ring $R=\mathbb{Z}_{2^t}$. Then, its units are $R^*=\{1,3,\ldots,2^t-1\}$ and its residue field is $\ff_2$. 
		In this case $-1\in U_{\ff_2,2}$ trivially since $-1=1$ in $\ff_2$, but 
		$$ -1 = 2^t-1 \not \in U_{R,2}. $$ 
		Indeed, the unique unit square of $R$ is $1$. In this way, we have that $G_{\Z_{2^t}}(2)$ is directed, but $G_{\ff_2}(2)$ is undirected.
		\hfill $\diamond$
	\end{exam}

	\section{The graphs $\mathcal{G}_R$ as graphs $G_R$, in odd characteristic} \label{sec: rel entre GRk y mathcalGRk}
	Our next goal is to express the graph $\mathcal{G}_R(k)$ for a given $k$ in terms of a graph $G_R(l)$ for certain well-determined $l$, thus reducing the study of the symmetrized graphs $\mathcal{G}_{R}$ to the graphs $G_R$.
	We will be able to do this in odd characteristic (or 2, but not even). We do not require here that $R$ be finite.

	\subsection{2-adic reductions of $k$ for $G_R(k)$} \label{subsec:4.1}
	In Section \ref{sec: reduction}, we studied a reduction for the integer $k$ in $G_R(k)$. Here we will give another reduction for $k$, obtained by reducing by powers of~$2$, that will be useful for the next subsection.

	We begin with the following elemental result, which will be used in the proof of the next proposition. 
	It is probably well-known, but we include the proof for completeness.
	\begin{lem} \label{lem: U2}
		Let $(R, \frak{m})$ be a commutative local ring of odd characteristic and let $u\in R^*$. 
		Then, $u^{2}=1$ if and only if $u=\pm 1$.
	\end{lem}
	
	\begin{proof}
		Notice that since the characteristic of $R$ is odd, we have that $2$ is a unit of $R$.
		Assume that $u\in R^*$ such that $u^{2}=1$, then $(u-1)(u+1)=0$.
		Notice that if we assume that $u-1$ and $u+1$ are both non-zero, then $u-1,\, u+1 \in \frak{m}$, and hence $2=(u+1)-(u-1)\in \frak{m}$ since $\frak{m}$ is an ideal. 
		But this is absurd since $2 \in R^*$. 
		Hence, either $u-1=0$ or else $u+1=0$, which implies that $u=\pm 1$ as asserted.
	\end{proof}

	In the sequel, we will make repeatedly use the number 
	\begin{equation} \label{eq: N v2k-v2q}
		N = v_{2}(k)-v_{2}(q-1)+1, 
	\end{equation}
	where $k$ is a positive integer and $q$ a prime power.
	
	Consider the quadratic map in $R$,
	$$ \Phi: R \rightarrow R, \qquad \Phi(u)=u^{2}. $$ 
	Using the previous two lemmas we can study the properties of the connection sets $U_{R,k}$ and the relationship between  $U_{R,\frac{k}{2}}$ and $U_{R,k}$ and the map $\Phi$ restricted to these sets.

	\begin{prop} \label{prop: Urk/2 = 2URk}
		Let $(R, \frak{m})$ be a commutative local ring of odd characteristic with $R/\frak{m}\cong \ff_q$ and let $k \in 2\N$ such that $v_{2}(k)\ge v_{2}(q-1)$. 
		Then, the quadratic map $\Phi$ restricted to $U_{R, \frac k2}$,
		$$ \Phi_{R,\frac k2}: U_{R,\frac{k}{2}} \rightarrow U_{R,k}, \qquad u \mapsto u^2, $$ 
		is surjective and we have the following:
		\begin{enumerate}[$(a)$]
			\item If $v_{2}(k)>v_{2}(q-1)$, then $\Phi_{R,\frac k2}$ is a bijection and hence $U_{R,\frac k2} = U_{R,k}$. 
			Indeed, we have 
			\begin{equation} \label{eq: URk2N}
				U_{R,\frac{k}{2^j}}=U_{R,k} \qquad (0\le j \le N-1).
			\end{equation} 
			
			\item If $v_{2}(k)=v_{2}(q-1)$, then $\Phi_{R, \frac k2}$ is a $2$-to-$1$ map and we have that 
			\begin{equation}\label{eq: URk2 = URk -URk}
				U_{R,\frac{k}{2}}=U_{R,k}\cup (-U_{R,k}).
			\end{equation}
			
			\item Moreover, if $R$ is finite it holds 
			$$|U_{R,k}| =  \begin{cases}
				\tfrac 12 |U_{R,\frac{k}{2^{N}}}| & \qquad \text{if $v_{2}(k)>v_{2}(q-1)$},\\[1.75mm] 
				\frac 12 |U_{R,\frac{k}{2}}| 	  & \qquad \text{if $v_{2}(k)=v_{2}(q-1)$}. 
			\end{cases}$$ 
		\end{enumerate}
	\end{prop}	
	
	\begin{proof}
		Fisrt, notice that if $x\in R^*$, then $x^{k}=(x^{\frac{k}{2}})^{2} \in \Phi(U_{R,\frac{k}{2}})$ and so the map $\Phi_{R, \frac k2}$ is surjective. 
		
		\smallskip 
		
		\noindent ($a$)	
		Assume that $v_{2}(k)> v_{2}(q-1)$. Then, $v_{2}(\frac{k}{2})\ge v_{2}(q-1)$ which implies that $|\ker \Phi|=1$. Thus, the map $\Phi$ is 1-1 and hence a bijection, as asserted. This clearly implies that 
		$$U_{R,\frac k2} = U_{R,k}.$$ 
		The identity $U_{R,\frac{k}{2^j}}=U_{R,k}$ for $j=1,\ldots, N-1$ follows from the previous identity by induction.
		
		\smallskip 
		
		\noindent ($b$) 	
		Now, assume that $v_{2}(k)=v_{2}(q-1)$. It is enough to show that $|\ker \Phi_{R, \frac k2}|=2$, since $\Phi$ is a group homomorphism.
		In this case $c\in \ker \Phi$ if and only if $c^{2}=1$ and, by Lemma \ref{lem: U2}, this in turn happens if and only if $c=\pm 1$. 
		Now, by Lemma \ref{lem: condition for -1 in URk, R local} we have that $\pm 1\in U_{R,\frac{k}{2}}$ which implies that $|\ker \Phi|=2$, as claimed.
		
		To see that equation \eqref{eq: URk2 = URk -URk} holds, it is enough that any element of $U_{R,\frac{k}{2}}$ belongs to $U_{R,k}$ or to $-U_{R,k}$, since $-1\in U_{R,\frac{k}{2}}$ and $U_{R,k}\subseteq U_{R,\frac{k}{2}}$. Notice that item ($a$) implies that $\Phi(U_{R,k})=U_{R,k}$, and since $-1\not \in U_{R,k}$, item ($c$) of Lemma \ref{lem: condition for -1 in URk, R local} implies that 
		$$ 
		U_{R,k} \cap (-U_{R,k})=\varnothing.
		$$	
		Moreover, since $(-1)^{2}=1$ we have that 
		$$ 
		\Phi(-U_{R,k})=\Phi(U_{R,k})=U_{R,k}=\Phi(U_{R,\frac{k}{2}}).
		$$
		Now, if $a\in U_{R,\frac{k}{2}}$ with $a\not \in U_{R,k}$, then $\Phi(a)\in U_{R,k}$ and the above equalities imply that there exist $b\in -U_{R,k}$ and $c\in U_{R,k}$ such that $\Phi(c)=\Phi(a)=\Phi(b)$.
		Since the map is $2$-to-$1$ when we restrict to $U_{R,\frac{k}{2}}$ and $a\not \in U_{R,k}$ then $a=b\in -U_{R,k}$. 
		Therefore, we obtain that 
		$$
		U_{R,\frac{k}{2}}=U_{R,k}\cup (-U_{R,k}),
		$$
		as desired.
		
		\smallskip 
		
		\noindent ($c$)
		If $R$ is finite, we clearly have $|U_{R,\frac{k}{2}}|=2 |U_{R,k}|$ in the case when $v_{2}(k)=v_{2}(q-1)$. 
		The remaining equality $|U_{R,\frac{k}{2^{N}}}|=2 |U_{R,k}|$ in the case $v_{2}(k)>v_{2}(q-1)$ follows from ($b$).
	\end{proof}
	
	As a consequence, under certain mild arithmetic conditions, we can divide $k$ by powers of 2, namely up to $2^{N-1}$.
	More precisely, we have the following.
	
	\begin{coro}
		Let $(R, \frak{m})$ be a commutative local ring of odd characteristic with residue field $R/\frak{m}\cong \ff_q$ and let $k \in 2\N$ such that $v_{2}(k)\ge v_{2}(q-1)$. 
		Then, we have that 
		\begin{equation} \label{eq: chain GRk = GRk2N-1}
			G_{R}(k) = G_{R}(\tfrac{k}{2}) = G_{R}(\tfrac{k}{2^2}) = \cdots = G_{R}(\tfrac{k}{2^{N-1}}),
		\end{equation}	 
		where $N$ is as in \eqref{eq: N v2k-v2q}.
	\end{coro}
	
	\begin{proof}
		It follows directly from item ($a$) of Proposition \ref{prop: Urk/2 = 2URk}.
	\end{proof}
	
	We now compare the two reductions at hand for the power $k$ in $G_R(k)$.
	\begin{rem}
		In Section \ref{sec: reduction}, we showed a reduction for $k$ in $G_R(k)$ for $R$ a finite local ring. In this section, we present another reduction of $k$ when $R$ is a local ring of odd characteristic. Therefore, if $R$ is a local finite ring of odd characteristic we have both reductions at disposal. In this case, by \eqref{eq: GRk=GRk'} and \eqref{eq: chain GRk = GRk2N-1} we have 
		$$ G_R(k) = G_R(k') = G_R(\tfrac{k}{2^{N-1}} ). $$
		In general, the first reduction is better, since $k' \mid \frac{k}{2^{N-1}}$. 
		Indeed, it is clear that 		 		
		$$ v_{p}(k') \le v_{p}(\tfrac{k}{2^{N-1}}) $$ 
		for any prime $p\neq 2$ and a simple computation allows us to conclude that 
		$$ v_{2}(k')\le v_{2}(\tfrac{k}{2^{N-1}}) $$ 
		as well.
		However, the number $N$ (and hence $\frac{k}{2^{N-1}}$) is very easy to compute, while the computation of the number $k'$ can be more involved.
	\end{rem}

	\subsection{The graphs $\mathcal{G}_R$ as graphs $G_R$}
	Now, we show that the graph $\mathcal{G}_R(k)$ is a graph $G_R(l)$ for a well-determined $l \le k$.
	
	First, by using the previous lemmas and proposition we show that the symmetrized connection set $T_{R,k}$ is the connection set $U_{R,\frac{k}{2^N}}$.
	
	\begin{lem} \label{lem: U=T}
		Let $(R, \frak{m})$ be a commutative local ring of odd characteristic with $R/\frak{m}\cong \ff_q$ and let $k \in 2\N$ such that $v_{2}(k)\ge v_{2}(q-1)$.
		Then, for $N$ as in \eqref{eq: N v2k-v2q} we have
		\begin{equation} \label{eq: TRk = URk/2t}
			T_{R,k}= U_{R,\frac{k}{2^N}} = U_{R,k} \cup (-U_{R,k}).
		\end{equation} 
	\end{lem}	
	
	\begin{proof}
		Let $\kappa$ be defined by
		$$ k= 2^{v_{2}(k)-v_{2}(q-1)} \kappa, $$
		so that $v_{2}(\kappa)=v_{2}(q-1)$. 
		On the other hand, by Lemma \ref{prop: Urk/2 = 2URk} we have that $U_{R,k}=U_{R,\kappa}$ and, hence, $T_{R,k}=T_{R,\kappa}$.
		
		Notice that $v_{2}(\frac{\kappa}{2})<v_{2}(q-1)$, which implies that $-1\in U_{R,\frac{\kappa}{2}}$ by Lemma \ref{lem: condition for -1 in URk, R local}.
		It is clear that $U_{R,\kappa} \subseteq U_{R,\frac{\kappa}{2}}$ and since $-1\in U_{R,\frac{\kappa}{2}}$,
		we obtain that $-U_{R,\kappa} \subseteq U_{R,\frac{\kappa}{2}}$ due to the fact that $U_{R,\frac{\kappa}{2}}$ is closed under multiplication. 
		Thus, we have that 
		$$
		T_{R,\kappa}=U_{R,\kappa}\cup (-U_{R,\kappa}) \subseteq U_{R,\frac{\kappa}{2}}.
		$$
		
		Hence, since $v_{2}(\kappa)=v_{2}(q-1)$, by Proposition \ref{prop: Urk/2 = 2URk} we have that 
		$$ U_{R,\kappa}\cup (-U_{R,\kappa}) = U_{R,\frac{\kappa}{2}}, $$ 
		which implies that $T_{R,k} = U_{R,\frac{\kappa}{2}} = U_{R,\frac{k}{2^N}}$ with $N=v_{2}(k)-v_{2}(q-1)+1$, as desired.
	\end{proof}
	
	We are finally in a position to state and prove the result showing that the graph $\mathcal{G}_{R}(k)$ is either $G_{R}(k)$ or $G_{R}(\tfrac{k}{2^N})$. 
	We state it in arithmetic terms, but the cases depend upon whether the graph $G_R(k)$ is directed or undirected.
	
	\goodbreak 
	
	\begin{thm} \label{teo: Gk and Gk2}
		Let $(R, \frak{m})$ be a commutative local ring with residue field $R/\frak{m} \cong \ff_{q}$ and let $k$ be a positive integer. 
		Then, we have:
		
		\begin{enumerate}[$(a)$]
			\item If  $R$ has characteristic $2$, then $\mathcal{G}_{R}(k)=G_{R}(k)$. \msk 
			
			\item If $R$ has odd characteristic, then 	
			\begin{equation*}
				\mathcal{G}_{R}(k)=
				\begin{cases}
					G_{R}(k)				& \qquad \text{if $v_{2}(k)< v_{2}(q-1)$, i.e.\@ when $G_R(k)$ is undirected}, \\[1.5mm]
					G_{R}(\tfrac{k}{2^N})   & \qquad \text{if $v_{2}(k)\ge v_{2}(q-1)$, i.e.\@ when $G_R(k)$ is directed},
				\end{cases}
			\end{equation*}	
			where $N$ is as in \eqref{eq: N v2k-v2q}. 
			\msk 
			
			\item If $k\mid q-1$ with $q$ odd and $-1\not \in U_{R,k}$, then $G_{R}(k)$ is directed and 
			\begin{equation}\label{eq: Gr Grsim}
				\mathcal{G}_{R}(k) = G_{R}(\tfrac{k}{2}) = \overset{_{\rightarrow}}{G}_{R}(k) \cup \overset{_{\leftarrow}}{G}_{R}(k)
			\end{equation}
			where $\overset{_{\rightarrow}}{G}_{R}(k)=G_{R}(k)$ and $\overset{_{\leftarrow}}{G}_{R}(k)$ is obtained from the digraph $\overset{_{\rightarrow}}{G}_{R}(k)$ by reversing all of its arrows.
		\end{enumerate}
	\end{thm}

	\begin{proof}
		\noindent $(a)$ 
		It is clear that if $-1\in U_{R,k}$, then $U_{R,k}=T_{R,k}$ since in this case $U_{R,k}=-U_{R,k}$ and hence $G_{R}(k)=\mathcal{G}_{R}(k)$. Thus, if $R$ has characteristic $2$, then
		$G_{R}(k)=\mathcal{G}_{R}(k)$, since $-1=1\in U_{R,k}$, which implies ($a$). 
		
		\smallskip 
		
		\noindent $(b)$ 
		On the other hand, if $R$ has odd characteristic the Lemma \ref{lem: condition for -1 in URk, R local} states that $-1\in U_{R,k}$ if and only $v_{2}(k)<v_{2}(q-1)$.
		
		Notice that if $-1\not \in U_{R,k}$, then $R$ must have odd characteristic and $k$ must be even, 
		by Lemma~\ref{lem: U=T} we obtain that $U_{R,\frac{k}{2^N}}=T_{R,k}$ where $N=v_{2}(k)-v_{2}(q-1)+1$ and hence $G_{R}(\frac{k}{2^N})=\mathcal{G}_{R}(k)$, as asserted. 
		
		Notice that $(b)$ in Lemma \ref{lem: condition for -1 in URk, R local} implies that $G_{R}(k)$ is undirected if and only if the inequality $v_2(k)<v_2(q-1)$ holds.
		
		\smallskip 
		
		\noindent $(c)$ 
		It is clear that if $k\mid q-1$ and $-1\not \in U_{R,k}$ with $q$ odd, then $v_{2}(k)=v_{2}(q-1)$. Thus, we obtain that $t=1$ and so by item ($b$) we have 
		$$
		\mathcal{G}_{R}(k)=G_{R}(\tfrac{k}{2}).
		$$
		On the other hand, by $(c)$ of Lemma \ref{lem: condition for -1 in URk, R local}, 
		we have that $U_{R,k}\cap (-U_{R,k}) = \varnothing$. 
		Also, by Lemma \ref{lem: U=T} we have that 
		$$ T_{R,k} = U_{R,\frac k2} = U_{R,k} \cup (-U_{R,k}).$$ 
		Therefore, Theorem 2.4 in \cite{PV19} implies the equality \eqref{eq: Gr Grsim}.
	\end{proof}
	
	Some additional comments about the theorem are worth giving.
	\begin{rem}
		From Theorem \ref{teo: Gk and Gk2} we conclude that, when $R$ is a local ring of odd characteristic, the study of the symmetrized graphs $\mathcal{G}_{R}$ is reduced to the one of the graphs $G_{R}$. 
	\end{rem}
	
	\begin{rem}
		Item ($c$) from the above theorem generalizes Theorem 3.5 in \cite{PV19} for GP-graphs. Indeed, taking $R=\ff_q$ in ($c$) of Theorem \ref{teo: Gk and Gk2}  we obtain Theorem 3.5 from \cite{PV19}.
	\end{rem}
	
	We now illustrate item ($a$) of the previous theorem.
	\begin{exam}
		Instances of commutative rings $R$ of characteristic 2 are the finite fields $\ff_{2^m}$, the polynomials $\Z_2[x_1,\ldots, x_k]$, the quotient rings $\ff_2[x]/(f(x))$ with $f$ a polynomial, the Boolean rings in which $x^2=x$ for every $x$, the ring of series $\ff_2[\![x]\!]$. For all these rings we have that $\mathcal{G}_R(k)=G_R(k)$.
		\hfill $\diamond$	
	\end{exam}
	
	The following example shows that the above theorem is not true when $R$ has even characteristic greater than $2$.
	\begin{exam} \label{exam: not true for Cal GRk}
		Consider the local ring $R=\mathbb{Z}_{8}$, with maximal ideal $\frak{m}=\{0,2,4,6\}$. 
		First, notice that there are only two 
		different graphs $G_{\mathbb{Z}_{8}}(k)$ up to isomorphism  for any $k \in \N$, namely 
		$G_{\mathbb{Z}_8}(1)$ and $G_{\mathbb{Z}_8}(2)$. 
		Indeed, by Proposition \ref{prop: reduction GRk GRk'}, for any $k=2^h\ell$ with $\ell$ odd we have that $k'=2^{\min\{2,h\}}$, since $\ell'=1$. Thus, we have 
		$$ G_{\Z_8}(j) = G_{\Z_8}(1), \qquad G_{\Z_8}(2j) = G_{\Z_8}(2) \qquad \text{and} \qquad G_{\Z_8}(4j) = G_{\Z_8}(4), $$
		for any odd $j$.

		Now, notice that $G_{\Z_8}(2)$ is the directed 8-cycle, since $a^{2}=1\pmod{8}$ for any $a\in \mathbb{Z}_{8}^*$ (and hence $a^4=1\pmod{8}$ for any $a\in \mathbb{Z}_{8}^*$),
		while $G_{\Z_8}(1)$ is the complete bipartite graph with 8 vertices 
		(see Example~\ref{exam: comp bip char 2}). 
		That is, 
		$$ G_{\Z_8}(1) = K_{4,4} \qquad \text{and} \qquad G_{\Z_8}(2)= G_{\Z_8}(4) = \vec{C}_{8}.$$ 
		On the other hand, we have that 
		$$ \mathcal{G}_{\Z_8}(2) = C_8 $$ 
		is the 8-cycle graph. Thus, the graph $\mathcal{G}_{\Z_8}(2)$ is not of the form $G_{\Z_8}(k)$ for any $k$. 
		\hfill $\diamond$
	\end{exam}

	\section{Blow-up decompositions of $G_R$ and $\mathcal{G}_R$} \label{sec: blowups}
	Here we show that, under the assumptions 
	$$ \text{$(R,\frak m)$ a \textsl{finite commutative local ring with unity} \quad and \quad $(k,|R|)=1$},$$
	the graphs $G_R(k)$ and $\mathcal{G}_R(k)$ are balanced blow-ups of the associated graphs $G_{R/\frak{m}}(k)$ and $\mathcal{G}_{R/\frak{m}}(k)$ defined over the residue field $R/\frak m$, respectively. 
	We point out that the hypothesis $(k,|R|)=1$ is crucial to get the blow-ups.

	\subsection{Blow-up decomposition for the graphs $G_R$} 
	As in the case $k=2$ shown by Liu-Zhou in \cite{LZ2}, 
	we recently showed in \cite{PV9} that there is a direct relationship between the graph $G_R(k)$ and the graph $G_{R/\frak{m}}(k)$ defined over the residue finite field $R/\frak{m}$.
	Namely, that one can recover $G_R(k)$ from $G_{R/\frak{m}}(k)$ either by performing a blow-up of order $m$ or by tensoring with $\mathring{K}_m$. 
	
	We now recall this result, using the notations of \S\ref{subsec: preliminares} in the Introduction. 
	
	\begin{thm}[{\cite[Theorem 2.3]{PV9}}] \label{teo: blowup GRk}
		Let $(R,\frak{m})$ be a finite commutative local ring with $m=|\frak{m}|$ and residue field $R/\frak{m} \simeq \ff_q$. If $k\in \mathbb{N}$ satisfies $(k,|R|)=1$, then
		\begin{equation} \label{eq: blowup R local}
			G_R(k) \simeq {G_{\ff_q}(k)}^{(m)} \simeq G_{\ff_q}(k) \otimes \mathring{K}_{m},  
		\end{equation} 
		In particular, $G_R(k)$ is $\frac{m(q-1)}{k'}$-regular where $k'=(k,q-1)$. 
	\end{thm}

	For instance, we have the following examples in different characteristics. We begin with characteristic 2. 
	
	\begin{exam} \label{exam: comp bip char 2}
		Let $(R, \frak{m})$ be a finite commutative local ring with characteristic $2$ and let $k \in 2\N+1$. Then,
		$$ G_{R}(k) \simeq \G(k,2)^{(|R|/2)} \simeq K_{\frac{|R|}{2},\frac{|R|}{2}}. $$
		That is, $G_{R}(k)$ is the complete bipartite graph with partitions of size $\tfrac 12  |R|$, 
		since in this case we have $R/\frak{m} = \mathbb{F}_{2}$, $|\frak{m}| = \frac 12 |R|$, and $\G(k,2)=K_2$. 
		
		In particular, for any $k\in \mathbb{N}$ odd we have
		$$	G_{\mathbb{Z}_{2^{\alpha}}}(k)\simeq K_{2^{\alpha-1},2^{\alpha-1}} .$$
		Notice that in this case, the graph has three different eigenvalues, namely $\pm 2^{\alpha -1}$ and $0$ and so it is strongly regular with parameters $srg(2^{\alpha},2^{\alpha-1},0,2^{\alpha-1})$.
		\hfill $\diamond$
	\end{exam}
	
	In odd characteristic, we have the next example.
	\begin{exam} \label{exam: mulipatito}
		Let $(R,m)$ be a local ring with prime residue field $\ff_p$ and let $k\in \N$ with $k$ coprime with $p$ and $p-1$, that is $(k,p(p-1))=1$. 
		Then, we obtain that 
		$$ G_{R}(k)=G_{\ff_p}(k)^{(|R|/p)}.$$ 
		We have that $G_{\ff_p}(k)=K_p$ (the complete graph with $p$ vertices), and so $G_R(k)$ is the complete multipartite graph $K_{p \times m}$.
		
		As a particular case when $k=1$, if $|R|=p^r$ we obtain that 
		$$ G_R(1) = G_{\ff_p}(1)^{(p^r/p)} = K_{p \times p^{r-1}}$$ 
		is the complete multipartite graph with $p$ independent sets of size $p^{r-1}$.
		\hfill $\diamond$
	\end{exam}

	We will need the following notation. An integer $n$ is a \textit{primitive divisor} of $p^m-1$ if $n\mid p^m-1$ and $n\nmid p^{a}-1$ for all $1\le a<m$. 
	For simplicity, as in our previous works \cite{PV7}, \cite{PV9}, \cite{PV18}, \cite{PV19}, we denote this fact by  
	\begin{equation} \label{eq: primitive divisor}
		n\dagger p^m-1.
	\end{equation} 
	Also, it is well-known that 
	\begin{equation} \label{eq: connectivity condition}
		\G(k,q) \text{ is connected} \qquad \Leftrightarrow \qquad n \dagger q-1,
	\end{equation}
	where $n$ is the regularity degree of $\G(k,q)$, that is 
	$$n=\frac{q-1}{k'} \qquad \text{ where } \qquad k'= (k,q-1).$$

	Now, we recall some structural consequences for the graph $G_{R}(k)$ when $R$ is a finite commutative local ring.

	\begin{coro}[{\cite[Corollary 2.5]{PV9}}] \label{-1RFq}
		Let $(R,\frak{m})$ be a finite commutative local ring with $m=|\frak{m}|$ and residue field $R/\frak{m} \simeq \ff_q$. 
		Let $k\in \mathbb{N}$ be such that $(k,|R|)=1$. 
		Then, 
		\begin{enumerate}[$(a)$]
			\item $-1\in U_{R,k}$ if and only if $-1\in U_{R/\frak m,k}$, 
			which occurs if and only if $q$ is even or else if $q$ is odd and $(k,q-1) \mid \frac{q-1}{2}$. \sk
			
			\item $G_R(k)$ is undirected if and only if $G_{R/\frak m}(k)$ is undirected. \sk 
			
			\item $G_{R}(k)$ is connected if and only if  $G_{R/\frak m}(k)$ is connected. This in turn happens if and only if $\frac{q-1}{(k,q-1)} \dagger q-1$. 
		\end{enumerate}
	\end{coro}

	\subsection{Blow-up decomposition for the symmetrized graphs $\mathcal{G}_R$} 
	In order to obtain the blow-up decomposition for the graph $\mathcal{G}_R(k)$, with $R$ local and $(k,|R|)=1$, we need first to recall some results previously obtained in \cite{PV9}.
	
	\begin{lem}[{\cite[Lemma 2.1]{PV9}}] \label{lem: m-cosets}
		Let $(R,\frak{m})$ be a finite commutative local ring with identity
		with residue field $R/\frak{m} \simeq \ff_q$. 
		Let $k\in \mathbb{N}$ be such that $(k,|R|)=1$.
		Then, we have: 
		\begin{enumerate}[$(a)$]
			\item For every $a\in R^*$ the map $g_a : a+\frak{m} \rightarrow a^{k}+\frak{m}$ given by $g_a(x)=x^k$ is a bijection. 
			\smallskip
			
			\item Let $a,b\in R$ be such that $b-a\in U_{R,k}$. 
			If $a\equiv c\pmod{\frak{m}}$ and $b\equiv d\pmod{\frak{m}}$, then we have $d-c\in U_{R,k}$. 
		\end{enumerate}
	\end{lem} 
	
	By changing ``arcs'' by ``edges'' in Proposition~2.2 from \cite{PV9}, we obtain the following result.
	We use the notation $E(\mathcal{G}_R)$ to denote the set of edges in a graph $\mathcal{G}_R$.
	
	\begin{prop} \label{Prop mod Local case}
		Let $(R,\frak{m})$ be a finite commutative local ring with identity and let $k\in \mathbb{N}$ be such that $(k,|R|)=1$. 
		If $ab \in E(\mathcal{G}_R(k))$ then $cd \in E(\mathcal{G}_R(k))$ for every $c\in a+\frak m$ and $d\in b+\frak m$. The vertices in each coset $a+\frak m$ of $\frak m$ are independent. \hfill $\square$
	\end{prop}
	
	Finally, we can state and prove the blow-up decomposition for the graphs $\mathcal{G}_R(k)$, in the notation of \S\ref{subsec: preliminares}. 
	
	\begin{thm} \label{teo: blowup mathcalGRk}
		Let $(R,\frak{m})$ be a finite commutative local ring with $|\frak{m}|=m$ and residue field $R/\frak{m} \simeq \ff_q$. 
		If $k\in \mathbb{N}$ satisfies $(k,|R|)=1$, then
		\begin{equation} \label{eq: blowup symmetrized GRk}
			\mathcal{G}_R(k) \simeq \mathcal{G}_{\ff_{q}}(k)^{(m)}.     
		\end{equation} 
		
		Furthermore, if $R$ has odd characteristic then 
		$$ \mathcal{G}_R(k) \simeq \begin{cases}
			{G_{\ff_q}(k)}^{(m)} 				& \qquad \text{if $v_{2}(k)< v_{2}(q-1)$}, \\[1mm]
			{G_{\ff_q}(\tfrac{k}{2^N})}^{(m)}   & \qquad \text{if $v_{2}(k)\ge v_{2}(q-1)$},
		\end{cases} $$
		where $N$ is as in \eqref{eq: N v2k-v2q}. 
		In particular, the graph $\mathcal{G}_R(k)$ is $\frac{m(q-1)}{k'}$-regular or $\frac{2m(q-1)}{k'}$-regular respectively, where $k'=(k,q-1)$.
	\end{thm}

	\begin{proof}
		Since $(k,|R|)=1$, by Proposition \ref{Prop mod Local case},
		$\mathcal{G}_{R}(k)$ is the balanced blow-up of order $m$ of $\mathcal{G}_{\ff_q}(k)$, where the independent sets are all the cosets of $\frak m$ in $R$, and hence \eqref{eq: blowup symmetrized GRk} holds. 
		
		Now, let us assume that $R$ has odd characteristic. 
		By Theorem \ref{teo: Gk and Gk2}, we have that 
		$$	\mathcal{G}_{\ff_q}(k)=
		\begin{cases}
			G_{\ff_q}(k) 				& \qquad \text{if $v_{2}(k)< v_{2}(q-1)$}, \\[1mm]
			G_{\ff_q}(\tfrac{k}{2^t})   & \qquad \text{if $v_{2}(k)\ge v_{2}(q-1)$}.
		\end{cases} 
		$$
		Since $(k,|R|)=(\frac{k}{2^t},|R|)=1$, Theorem \ref{teo: blowup GRk} implies that
		$$ G_{R}(k) \simeq G_{\ff_q}(k)^{(m)} 
		\qquad \text{and} \qquad  
		G_{\ff_q}(\tfrac{k}{2^t}) \simeq G_{\ff_q}(\tfrac{k}{2^t})^{(m)}, $$ 
		whose independent sets are all the cosets of $\frak m$ in $R$.
		Thus, by \eqref{eq: blowup symmetrized GRk} we have
		$$	\mathcal{G}_R(k) \simeq \mathcal{G}_{\ff_{q}}(k)^{(m)} \simeq 
		\begin{cases}
			{G_{\ff_q}(k)}^{(m)} 				& \qquad \text{if $v_{2}(k)< v_{2}(q-1)$}, \\[1mm]
			{G_{\ff_q}(\tfrac{k}{2^t})}^{(m)}   & \qquad \text{if $v_{2}(k)\ge v_{2}(q-1)$},
		\end{cases}  
		$$
		as asserted. 
		
		Finally, recall that the regularity degree of $\mathcal{G}_{R}(k)$ is given by $|T_{R,k}|$. By Lemmas \ref{prop: Urk/2 = 2URk} and \ref{lem: U=T} we have that
		$$ |T_{R,k}|=
		\begin{cases}
			|U_{R,k}|  & \qquad \text{if $v_{2}(k)< v_{2}(q-1)$}, \\[1mm]
			2|U_{R,k}| & \qquad \text{if $v_{2}(k)\ge v_{2}(q-1)$}.
		\end{cases} $$
		Therefore, by Theorem \ref{teo: blowup GRk} we have that $|U_{R,k}|=\frac{m(q-1)}{k'}$, where $k'=(k,q-1)$, which implies the last assertion in the statement.
	\end{proof}

	As a direct consequence of the above theorem, we obtain the following useful result.
	
	\begin{coro} \label{coro: Local case mathcal G}
		Let $(R,\frak{m})$ be a finite commutative local ring with residue field $R/\frak{m} \simeq \ff_q$. 
		If $R$ has odd characteristic and $k\in \mathbb{N}$ satisfies $(k,|R|)=1$, then
		\begin{equation} \label{eq: mathcalGRk = casos GRks}
			\mathcal{G}_R(k) \simeq \begin{cases}
				{\G(k',q)}^{(m)} 			& \qquad \text{if $v_{2}(k)< v_{2}(q-1)$}, \\[1mm]
				{\G(\tfrac{k'}{2},q)}^{(m)} & \qquad \text{if $v_{2}(k)\ge v_{2}(q-1)$},
			\end{cases}
		\end{equation}
		where $k'=(k,q-1)$ and $m=|\frak{m}|$.
	\end{coro}
	
	\begin{proof}
		By Theorem \ref{teo: blowup mathcalGRk}, if $t=v_2(k)-v_2(q-1)+1$ then we have that
		$$ \mathcal{G}_R(k)\simeq 
		\begin{cases}
			{\G(k,q)}^{(m)} 			 & \qquad \text{if $v_{2}(k)< v_{2}(q-1)$}, \\[1mm]
			{\G(\tfrac{k}{2^t},q)}^{(m)} & \qquad \text{if $v_{2}(k)\ge v_{2}(q-1)$}.
		\end{cases} 
		$$
		Since $\G(k,q)=\G(k',q)$ with $k'=(k,q-1)$, by \eqref{eq: Gkq red}, it is enough to show that 
		\begin{equation} \label{eq: Gk2t=Gk'2}
			\G(\tfrac{k}{2^t},q) = \G(\tfrac{k'}{2},q).	
		\end{equation}

		Thus, let $U_\ell=\{x^{\ell}:x\in \mathbb{F}_q^*\}$, then $|U_{\ell}|$ is the regularity degree of $\G(\ell,q)$ and $U_{\ell}$ is the unique subgroup of $\ff_{q}^*$ of size $|U_{\ell}|$. Recall that in general, since $\ff_{q}^*$ is cyclic of order $q-1$, then $|U_{\ell}|=\frac{q-1}{(\ell,q-1)}$.
		
		By Theorem \ref{teo: blowup mathcalGRk}, we have that $\mathcal{G}_R(k)$ is $\frac{2m(q-1)}{k'}$-regular when $v_2(k)\ge v_2(q-1)$, this implies that $\G(\frac{k}{2^{t}},q)$ is $\frac{2(q-1)}{k'}$-regular which implies that $|U_{\frac{k}{2^{t}}}|=\frac{2(q-1)}{k'}$, but the set $U_{\frac{k'}{2}}$ has also the same size which implies \eqref{eq: Gk2t=Gk'2}.
		In this way, we finally obtain \eqref{eq: mathcalGRk = casos GRks}, as desired.
	\end{proof}
	
	By taking into account that the balanced blow-up $G^{(m)}$ of a graph is connected (bipartite) if and only if $G$ is connected (bipartite), we obtain the following result.
	
	\goodbreak 
	
	\begin{coro} \label{coro RFq conn}
		Let $(R,\frak{m})$ be a finite commutative local ring with $m=|\frak{m}|$ and residue field $R/\frak{m} \simeq \ff_q$. 
		Let $k\in \mathbb{N}$ be such that $(k,|R|)=1$. 
		Then, we have:
		\begin{enumerate}[$(a)$]
			\item $\mathcal{G}_{R}(k)$ is connected if and only if $\mathcal{G}_{\ff_q}(k)$ is connected. \sk 
			
			\item $\mathcal{G}_{R}(k)$ is bipartite if and only if $\mathcal{G}_{\ff_q}(k)$ is bipartite.
		\end{enumerate}
	\end{coro}
	
	However, when $(k,|R|)>1$, the same phenomenon does not occur, as the following example shows. 
	
	\begin{exam}
		Consider the local ring $R=\Z_8$ with maximal ideal $\frak{m}=\{0,2,4,6\}$, hence $\Z_8/\mathfrak{m} = \ff_2$.
		Notice that $\mathcal{G}_{\Z_8/\frak{m}}(k)=K_{2}$ for any $k\in \mathbb{N}$ and, so, $\mathcal{G}_{\Z_8/\frak{m}}(k)^{(4)}\simeq K_{4,4}$ for any $k\in \N$. 
		We have that $\mathcal{G}_{\Z_8}(2)=C_8$ (see Example \ref{exam: not true for Cal GRk}) and hence
		$$ C_8 \simeq \mathcal{G}_{\Z_8}(2) \not \simeq  \mathcal{G}_{\Z_8/\frak{m}}(2)^{(4)} \simeq K_{4,4},$$
		since these graphs have different regularity degrees (2 and 4, respectively).
		\hfill $\diamond$
	\end{exam}

	\section{Product decompositions for $G_{R}$ and reduction} \label{sec: product decompositions for GRk}
	In this and the next section we will extend the results in Theorem \ref{teo: blowup GRk} for the graphs $G_{R}(k)$ and $\mathcal{G}_{R}(k)$ when $R$ is not a local ring. More precisely, in the case of an arbitrary finite commutative ring $R$, by using the Artin decomposition of $R$ we will provide Kronecker product decompositions of these graphs. When $(k,|R|)=1$, we can also express these graphs as blow-ups of the corresponding graph associated to the reduced ring $R_{red}$.

	\subsection{Artin decomposition of $R$ and Kronecker product decomposition of $G_R$}
	We begin by recalling the well-known Artin's structure theorem asserting that for a finite commutative ring with identity $R$ we have the Artin's decomposition 
	\begin{equation} \label{artin desc}
		R \simeq  R_1 \times \cdots \times R_s
	\end{equation} 
	with $s \in \N$, where each $R_i$ is a local ring with unique maximal ideal $\frak{m}_i$. From now on, we set the notations
	\begin{equation} \label{notations R}
		r=|R|, \qquad r_i=|R_i|, \qquad m_i=|\frak{m}_i| \qquad \text{and} \qquad R_i/\frak{m}_i \simeq \ff_{q_i},
	\end{equation}
	for $i=1,\ldots,s$. In this way, $r=r_1 \cdots r_s$ and $q_i=\tfrac{r_i}{m_i}$.
	Moreover, we also have the corresponding decomposition for the group of units:
	\begin{equation} \label{artin desc units}
		R^* \simeq R_1^* \times \cdots \times R_s^*.
	\end{equation}
	Thus, the ring $R$ is local if and only if $s=1$. 
	
	Due to this Artin's decomposition, it will be very convenient for us in the sequel to use the following notation:
	\begin{equation} \label{eq: Is}
		I_s = \{1,2,\ldots, s\}.
	\end{equation}
	
	Recall that the \textit{nilradical} $\mathcal{N}_R$ of $R$ is the intersection of the prime ideals of $R$ (equivalently, the intersection of maximal ideals when $R$ is finite). It can be shown that if $R=R_1\times \cdots \times R_s$ is as above, then 			
	\begin{equation} \label{eq: nilradical}
		\mathcal{N}_R = \frak{m}_1 \times \cdots \times \frak{m_s} \qquad \text{and thus} \qquad 
		|\mathcal{N}_R|=m=m_1\cdots m_s. 
	\end{equation}
	
	We begin by showing a Kronecker product decomposition of the graphs $G_R(k)$ induced by the Artin's decomposition of the ring $R$.
	\begin{prop} \label{prop: Kronecker GRk}
		Let $R$ be a finite commutative ring with identity with Artin's decomposition as in \eqref{artin desc}.
		Then, for any $k\in \mathbb{N}$ we have that 
		\begin{equation} \label{eq: GRk tensor GRki's}
			G_{R}(k) \simeq G_{R_1}(k) \otimes \cdots \otimes G_{R_s}(k). 
		\end{equation}
	\end{prop} 
	
	\begin{proof}
		The Artin's decomposition we have $R=R_1 \times \cdots \times R_s$ and $R^*= R_{1}^* \times \cdots \times R_{s}^{*}$
		which implies  
		$$ U_{R,k}= U_{R_1,k}\times \cdots \times U_{R_s,k}. $$
		We have that 
		$(a_1,\ldots,a_{s})$ and $(b_1,\ldots, b_s)$ are neighbors in $G_{R}(k)$ if and only if
		$$(b_1,\ldots, b_s)-(a_1,\ldots,a_{s})=(x_1,\ldots,x_s)^k$$ 
		where $x_{i}\in R_{i}^*$, i.e.\@ $b_{i}-a_{i}=x_{i}^k$ for all $i \in I_s$. Hence we get \eqref{eq: GRk tensor GRki's}.
	\end{proof}

	\subsection*{Reduced graphs and reduced rings}
	There are notions of reduced rings and reduced graphs. Fortunately, for the graphs $G_R(k)$ these notions are related in a natural way.
	
	A ring $R$ is \textit{reduced} if it has no non-zero nilpotent elements. 
	We recall that a finite ring $R$ is reduced if and only if it is isomorphic to a product of finite fields 
	$$R=\ff_{q_1} \times \cdots \times \ff_{q_{s}}$$
	(and this is of course its Artin's decomposition).
	The reduced ring of a ring $R$ is 
	$$ R_{red}=R/\mathcal{N}_R, $$ 
	where $\mathcal{N}_R=\{ x \in R: x^n=0 \text{ for some } n \in \N\}$ is the nilradical of $R$.
	
	Given a graph $\G=(V,E)$ , directed or undirected, we have the equivalence relation $\equiv$ on the vertex set $V(\G)$ given by
	$$
	a \equiv b \qquad \Leftrightarrow \qquad N(a)=N(b),
	$$
	where $N(a)=\{v\in V(\G): av \in E(\G)\}$ or 
	$N(a)=\{v\in V(\G): \vec{av} \in E(\G) \text{ or } \vec{va} \in E(\G)\}$, in the undirected or directed case, respectively.
	The \textit{reduced graph} of $\G$, denoted by $\G_{red}$ is the graph with vertex set the equivalence classes $[a]$ under the above equivalence relation such that two classes $[a]$ and $[b]$ are neighbors if and only if $a$ and $b$ are neighbors in $\G$ (with the corresponding orientation in the directed case), that is 
	$$ [a] \sim [b] \quad \text{ in } \G_{red} \qquad \Leftrightarrow \qquad a \sim b \quad \text{ in } \G. $$ 
	In general, the graph $\G$ is reduced if $\G_{red}\simeq \G$, i.e.\@ if $\#[a]=1$ for any $a \in V(\G)$.
	
	We begin with the following lemma stating that generalized Paley graphs $\G(k,q)$ are always reduced.
	
	\begin{lem}\label{lem GP are red}
		Let $q=p^{m}$, with $p$ prime and let $k\in \mathbb{N}$ such that $k\mid q-1$.
		Then, the graph $\G(k,q)$ is reduced.
	\end{lem}

	\begin{proof}
		Since the graph $\G(k,q)$ is vertex-transitive, it is enough to show that $\#[0]=1$. 
		Thus, it is enough to see that if $b \in \ff_q$ is such that $b + U_{k}=U_{k}$, where $U_{k}=\{x^k: x\in \ff_q^*\}$, then $b=0$.
		
		Let $b \in \ff_q$ such that $b + U_k =U_k$, then 
		$$
		\sum_{a\in U_k}a = \sum_{a\in \ff_q}(a+b) = \big( \sum_{a\in \ff_q}a \big) +  b \, |U_k|,
		$$
		which implies that $b\, |U_k|=0$, and so $b =0$ or else $p\mid |U_k|$. However, notice that $p\nmid |U_k|$ since 
		$|U_k|$ divides $q-1$ and $(p,q-1)=1$, hence $b =0$.
		Therefore, we obtain that $\G(k,q)$ is reduced, as asserted. 
	\end{proof}
	
	In particular, in the case that $R$ is reduced, we can give a slightly better decomposition 
	(i.e., one in terms of integers $k_i \le k$ for each $i \in I_s$).
	
	\begin{prop} \label{Nonlocal case red}
		Let $R$ be a finite reduced commutative ring with identity, with Artin's decomposition
		$R \simeq \ff_{q_1}\times \cdots \times \ff_{q_s}$. 
		For any $k\in \mathbb{N}$ we have 
		\begin{equation} \label{eq: GRk Kronecker GRki's}
			G_{R}(k) \simeq \G(k_1,q_1) \otimes \cdots \otimes \G(k_s,q_s)
		\end{equation}
		where $k_i=(k,q_i-1)$ for all $i=1,\ldots,s$.
	\end{prop}	
	
	\begin{proof}
		From Theorem \ref{prop: Kronecker GRk}, we have that $G_{R}(k) = G_{\ff_{q_1}}(k) \otimes \cdots \otimes G_{\ff_{q_s}}(k)$.  
		Since $G_{\ff_q}(k')=\G((k',q-1),q)$ holds in general, we have that $G_{\ff_{q_i}}(k) = \G(k_i,q_i)$ for each $i \in I_s$ by \eqref{eq: GPkq} and therefore \eqref{eq: GRk Kronecker GRki's} holds.
	\end{proof}
	
	As a direct consequence of Lemma \ref{lem GP are red} and the above proposition, we have the following result.
	
	\begin{coro}\label{coro Gred Rred}
		Let $R$ be a finite reduced commutative ring with identity, with Artin's decomposition
		$R \simeq \ff_{q_1}\times \cdots \times \ff_{q_s}$. 
		For any $k\in \mathbb{N}$, the graph $G_{R}(k)$ is reduced.
	\end{coro}
	
	\begin{proof}
		By definition of the Kronecker product, the Kronecker product of reduced graphs is reduced. By Proposition \ref{Nonlocal case red}, the graph $G_{R}(k)$ is isomorphic to the Kronecker product of $\G(k_i,q_i)$'s and by Lemma \ref{lem GP are red} these graphs are all reduced. Therefore, the graph $G_{R}(k)$ is reduced, as asserted. 
	\end{proof}

	\subsection{Some decompositions for $G_R(k)$ when $(k,|R|)=1$} \label{sec: 3.1}
	As a direct consequence of Theorem~\ref{teo: blowup GRk} and Proposition \ref{prop: Kronecker GRk}, 
	under the assumption $(k,|R|)=1$ we will obtain a better decomposition of $G_{R}(k)$, that is one in terms of Kronecker product of generalized Paley graphs which are defined over finite fields (recall that $G_{\ff_q}(k) = \G(k,q)$ and $K_q = \G(1,q)$).
	Moreover, the graph $G_{R}(k)$ can be obtained as a blow-up of the graph $G_{R_{red}}(k)$ associated to the reduced ring $R_{red}$.

	\begin{thm} \label{teo: GRk tensor GRki's caso no local}
		Let $R$ be a finite commutative ring with identity with Artin's decomposition $R=R_1 \times \cdots \times R_s$
		and let $k\in \mathbb{N}$. 
		If $(k,|R|)=1$, then 
		\begin{equation} \label{eq: GRK tensor GRki for R non local}
			G_{R}(k) \simeq \big( \bigotimes_{i=1}^s \G(k_i,q_i) \big) \otimes \mathring{K}_m \simeq G_{R_{red}}(k) \otimes \mathring{K}_m  \simeq (G_{R_{red}}(k))^{(m)} ,
		\end{equation}
		where $k_i=(k,q_i-1)$ and $m=|\mathcal{N}_R|$. 
	\end{thm}
	
	\begin{proof}
		By Theorem \ref{prop: Kronecker GRk} we have that $G_{R}(k)\simeq  G_{R_1}(k)\otimes \cdots \otimes G_{R_s}(k)$.
		Since $(k,|R|)=1$ and $|R|=|R_1|\cdots|R_s|$ we have that $(k,|R_i|)=1 $ for all $i=1,\ldots,s$.
		By Theorem \ref{teo: blowup GRk}, since each $R_i$ is a local ring, we have that
		$$ G_{R_i}(k)\simeq  G_{\ff_{q_i}}(k) \otimes \mathring{K}_{m_i}$$ 
		for all $i=1,\ldots,s$.
		Now, by taking into account that in general one has 
		$$ \mathring{K}_{n_1}\otimes \mathring{K}_{n_2}=\mathring{K}_{n_{1}n_{2}},$$
		from \eqref{eq: GRk tensor GRki's} and the fact that the $\otimes$ product is associative and commutative, we get
		$$ G_{R}(k) \simeq  \big( \bigotimes_{i=1}^s G_{\ff_{q_i}}(k) \big) \otimes \mathring{K}_m.$$
		As in the proof of the previous corollary, we have that $G_{\ff_{q_i}}(k) = \G(k_i,q_i)$ for each $i \in I_s$.
		So, we have obtained the first equivalence in \eqref{eq: GRK tensor GRki for R non local}. The second equivalence in 
		\eqref{eq: GRK tensor GRki for R non local} is a consequence of Proposition \ref{Nonlocal case red}, since $R_{red}=\ff_{q_1}\times \cdots \times \ff_{q_s}$ in this case. For the third equivalence we use \eqref{eq: blowup}. The last equivalence it follows from the definition of balanced blow-up.
	\end{proof}
	
	Notice that by \eqref{eq: blowup} and the commutative and associative properties of the Kronecker product, the expressions in \eqref{eq: GRK tensor GRki for R non local} can be put in terms of blow-ups as follows
	\begin{equation} \label{eq: blowups for GRk}
		G_{R}(k) \simeq (G_{\ff_{q_i}}(k))^{(m_1)} \otimes \cdots \otimes (G_{\ff_{q_s}}(k))^{(m_s)} \simeq (G_{R_{red}}(k))^{(m)}.
	\end{equation}
	This expression will be used in Theorem \ref{teo: GR bip}.

	\subsection*{Blow-ups and reductions} Notice that if we have a reduced graph $G$, then 
	\begin{equation}\label{eq bup red inv}
		(G^{(m)})_{red}\simeq G.
	\end{equation}
	So, as a consequence of Corollary \ref{coro Gred Rred} and Theorem \ref{teo: GRk tensor GRki's caso no local}, we have the following pleasing result stating that, under the assumption that $k$ is coprime with $|R|$, the reduction of the graph $G_R(k)$ is the $k$-th unitary Cayley graph associated to the reduced ring $R_{red}$.
	
	\begin{thm} \label{thm: GR_red=G_Rred}
		Let $R$ be a finite commutative ring with identity 
		and let $k\in \mathbb{N}$. 
		If $(k,|R|)=1$, then 
		\begin{equation} \label{eq: GR_red=G_Rred}
			\big(G_{R}(k)\big)_{red} \simeq G_{R_{red}}(k).
		\end{equation}	
		
	\end{thm}
	
	\begin{proof}
		By Theorem \ref{teo: GRk tensor GRki's caso no local}, we have that
		$$
		G_{R}(k)\simeq \big(G_{R_{red}}(k)\big)^{(m)}.
		$$
		Now, by Corollary \ref{coro Gred Rred}, we have that $G_{R_{red}}(k)$ is reduced and by \eqref{eq bup red inv} we obtain that 
		$$
		\big(G_{R}(k)\big)_{red}\simeq \big(G_{R_{red}}(k)\big)_{red}\simeq G_{R_{red}}(k),
		$$ 
		as desired.
	\end{proof}

	\subsubsection*{Prime graphs}
	We now make some comments relating reduced and connected graphs (treated in this work) with prime graphs.
	An \textit{homogeneous set} in a graph $\G$ is a set of vertices $X \subset V(\G)$ such that every vertex in $V(\G) \smallsetminus X$ is adjacent to either all or none of the vertices in $X$. 
	Note that $X=V(\G)$ and $X=\{v\}$ for any $v \in V(\G)$ are  homogeneous sets. An homogeneous set $X$ is \textit{non-trivial} if $2 \le |X| < |V(\G)|$.
	A graph is called \textit{prime} if it has no non-trivial homogeneous sets.
	Notice that by the very definitions, if $\G$ is any graph, we have that:
	\begin{itemize}
		\item $\G$ is prime $\Rightarrow$ $\G$ is reduced. \sk 
		
		\item $\G$ is prime $\Rightarrow$ $\G$ is connected and anticonnected (i.e.\@ connected complement).
	\end{itemize}
	Thus, if $\G$ is not reduced, or not connected or not antinconnected, then $\G$ is not prime.
	
	\begin{rem} \label{rem: primeness}
		Prime Cayley graphs $X(G, S)$ were studied in \cite{Minac}. In particular, \S 4.3 deals with unitary Cayley graphs over a ring $R$, that is $X(R,R^*)= G_R(1)$.  
		In a sequel, Nguyen and Tân (\cite{Nguyen}) studied the primeness of $p$-th unitary Cayley graphs over a commutative rings $G_R(p)$, for a prime power $p$. In fact, they give nice characterizations for the graph $G_R(p)$ to be prime for $R$ a finite commutative ring $R$ and $p$ a prime.
	\end{rem}

	\section{Product decomposition for the graphs $\mathcal{G}_{R}$} \label{sec: products of symmetrized GRks}
	Here we mimic the results obtained in the previous section for the graphs $G_R$, to the symmetrized graphs $\mathcal{G}_{R}$, where $R$ is an arbitrary finite commutative ring with identity.
	
	This undirected case is more technical since for the connection sets of $\mathcal{G}_{R}(k)$ we have that in general, 
	$$ T_{R_1 \times R_2,k} \ne T_{R_1,k} \times T_{R_2,k},$$ 
	as opposed to the previous case where we have that $U_{R_1 \times R_2,k} = U_{R_1,k} \times U_{R_2,k}$.
	
	We first give a natural necessary and sufficient condition for $-1$ to be in $U_{R,k}$.
	\begin{lem} \label{lem: GRk undirected conditions}
		Let $R$ be a finite commutative ring with identity with Artin's decomposition $R=R_1 \times \cdots \times R_s$ 
		and let $k\in \mathbb{N}$. Then, 
		$$ -1 \in U_{R,k} \qquad \Leftrightarrow \qquad -1\in U_{R_j,k} \quad \text{for every $j \in I_s$}. $$
	\end{lem}
	
	\begin{proof}
		Assume first that $-1\in U_{R,k}$.
		The Artin's decomposition \eqref{artin desc} implies that 
		\begin{equation} \label{Srk}
			U_{R,k} = U_{R_1,k}\times \cdots\times U_{R_s,k}.
		\end{equation} 	
		The above equality and the hypothesis $-1\in U_{R,k}$ imply that $-1\in U_{R_i,k}$ for all $i \in I_s$. 
		In fact, $-1\in R$ corresponds to $(-1,\ldots,-1) \in R_1 \times \cdots \times R_s$ in the Artin's decomposition.
		
		Conversely, if $-1\in U_{R_j,k}$ for any $j \in I_s$, then there exists $a_{i}\in R_i^{*}$ such that $a_{i}^{k}=-1$ in $R_{i}$ for all $i\in I_s$. 
		Using the Artin's decomposition for the units in \eqref{artin desc units}
		we have that $a=(a_1,\ldots,a_{s})\in R^*$ and 
		$$ a^{k} = (a_1^k,\ldots,a_{s}^k) = (-1,\ldots,-1)=-1. $$
		Hence, we obtain that $-1\in U_{R,k}$, as asserted. 
	\end{proof}
	
	In the same way as was done by Liu and Zhou in \cite{LZ2}, we will give a necessary and sufficient condition for undirected unitary Cayley graphs of $k$-th powers to have a Kronecker product decomposition. 
	We will need the following base case first.
	
	\begin{lem} \label{lem: GR1xR2 = GR1 x GR2}
		If $R_1$ and $R_2$ are finite commutative rings, then 
		\begin{equation} \label{eq:  GR1k x GR2k}
			\mathcal{G}_{R_1\times R_2}(k) = \mathcal{G}_{R_1}(k) \otimes \mathcal{G}_{R_2}(k) \quad  
			\Leftrightarrow \quad -1 \in U_{R_j,k}  \text{ for at least one $j\in \{1,2\}$}.
		\end{equation}
	\end{lem}
	
	\begin{proof}
		Notice that on the one hand, we obtain 
		$$ 
		T_{R_1,k} \times T_{R_2,k} = \{(\pm a^k, \pm b^k),(\pm a^k, \mp b^k): a\in R_1^{*}, b\in R_2^{*} \},
		$$
		while on the other hand, since $T_{R_1\times R_2,k} = (U_{R_1,k} \times U_{R_2,k}) \cup - (U_{R_1,k} \times U_{R_{2},k})$, we get that  
		$$ 
		T_{R_1 \times R_2,k} = \{\pm (a^k,b^k): a\in R_{1}^{*}, b\in R_{2}^*\} \subseteq T_{R_1,k} \times T_{R_2,k}. 
		$$
		We have that $\mathcal{G}_{R_1 \times R_2,k} = X(R_1 \times R_2, T_{R_1 \times R_2,k})$ and, from the definitions of $\mathcal{G}_{R_1}(k)$, $\mathcal{G}_{R_2}(k)$ and $\mathcal{G}_{R_1}(k)\otimes\mathcal{G}_{R_2}(k)$, it follows that 
		\begin{gather*}
			\mathcal{G}_{R_1}(k) \otimes \mathcal{G}_{R_2}(k) = X(R_1\times R_2, T_{R_1,k} \times   T_{R_2,k}). 
		\end{gather*}
		Therefore, 
		$$ \mathcal{G}_{R}(k) = \mathcal{G}_{R_1}(k) \otimes \mathcal{G}_{R_2}(k) \qquad \Leftrightarrow \qquad T_{R_1\times R_2,k} = T_{R_1,k} \times T_{R_2,k}. $$
		
		Now, if $-1\not\in U_{R_1,k}$ and $-1\not\in U_{R_2,k}$, then $(1,-1), (-1,1) \in T_{R_1,k} \times T_{R_2,k}$, but they are not elements of $T_{R_1 \times R_2,k}$, and hence $\mathcal{G}_{R_1\times R_2}(k)\neq \mathcal{G}_{R_1}(k)\otimes \mathcal{G}_{R_2}(k)$. Therefore, the equality of graphs $\mathcal{G}_{R_1\times R_2}(k) = \mathcal{G}_{R_1}(k)\otimes \mathcal{G}_{R_2}(k)$ implies that $-1 \in U_{R_1,k}$ or $-1 \in U_{R_2,k}$.
		
		To see the converse of \eqref{eq:  GR1k x GR2k}, suppose that $-1$ belongs to $U_{R_i,k}$ for at least one $i=1,2$. 
		Without loss of generality, we may suppose that $-1\in U_{R_1,k}$, and hence $-1=c^{k}$ for some $c\in R_{1}^{*}$.
		Then, for any
		$(a,b)\in (R_1\times R_2)^*$ and $t_1,t_2 \in \{0,1\}$ 
		we have
		$$ \big( (-1)^{t_1}a^k, (-1)^{t_2} b^{k} \big)= (-1)^{t_2} \big( (-1)^{t_1-t_2} a^k,b^k\big) = (-1)^{t_2} \big( (c^{t_1-t_2} a)^k,b^k\big).$$
		Hence $T_{R_1,k} \times T_{R_2,k} \subseteq T_{R_1\times R_2,k}$. In this way, 
		$T_{R_1,k} \times T_{R_2,k} = T_{R_1 \times R_2,k}$ 
		and hence $\mathcal{G}_{R_1\times R_2}(k) = \mathcal{G}_{R_1}(k) \otimes \mathcal{G}_{R_2}(k)$, as desired. 
	\end{proof}
	
	We are now ready to state and prove the general result.
	The proof is essentially the same as in the case $k=2$ of Theorem 2.5 of \cite{LZ2}, but we give it for the ease of the reader and completeness.
	\begin{thm} \label{thm: conditions for mathcalGRk Kronecker}
		Let $R$ be a finite commutative ring with identity with Artin's decomposition $R=R_1 \times \cdots \times R_s$. 
		If $k\in \mathbb{N}$, then 
		\begin{equation} \label{eq: GRK tensor no-local R}
			\mathcal{G}_{R}(k) \simeq \bigotimes_{i=1}^s \mathcal{G}_{R_i}(k) \qquad \Leftrightarrow \qquad -1\not \in U_{R_j,k} \quad \text{for at most one $j\in I_s$}.
		\end{equation}
	\end{thm}
	
	\begin{proof}
		We assume that $s\ge 3$, since the result is trivial for $s = 1$ and the case $s=2$ is Lemma \ref{lem: GR1xR2 = GR1 x GR2} with $R_1, R_2$ local rings. 
		
		Thus, assume that $s \ge 3$ and suppose that the result holds for any $\ell \le s$. 
		By induction, it is enough to prove the result for $s+1$. So, fix 
		$$ R= R_1 \times \cdots \times R_s \times R_{s+1}$$ 
		the Artin's decomposition of $R$. 
		
		\goodbreak 
		
		Assume that there is at most one index $j \in I_{s+1}$ such that $-1\not\in U_{R_j,k}$. 
		In this way, if $-1 \not \in U_{R_{s+1},k}$ then $-1 \in U_{R_j,k}$ for all $j \in I_s$, and this implies that $-1\in U_{R_{1,k} \times \cdots \times R_{s},k}$, by Lemma~\ref{lem: GRk undirected conditions}. 
		On the other hand, if $-1\not \in U_{R_{j},k}$ for some $j \in I_s$ then $-1\in U_{R_{s+1},k}$.	
		Hence, we have seen that either 
		$-1\in U_{R_1\times \cdots \times R_{s},k}$ or else $-1\in U_{R_{s+1},k}$ and, thus, by \eqref{eq:  GR1k x GR2k} we get 
		$$\mathcal{G}_{R_1\times \cdots \times R_{s+1}}(k)= \mathcal{G}_{R_1\times \cdots \times R_{s}}(k)\otimes \mathcal{G}_{R_{s+1}}(k).$$
		By the inductive hypothesis we have that $\mathcal{G}_{R_1\times \cdots \times R_{s}}(k)= \mathcal{G}_{R_1}(k) \otimes\cdots \otimes\mathcal{G}_{R_{s}}(k)$ and, therefore,
		$$\mathcal{G}_{R_1\times \cdots \times R_{s+1}}(k)= \mathcal{G}_{R_1}(k)\otimes \cdots \otimes \mathcal{G}_{R_{s+1}}(k),$$
		as we wanted to show.
		
		Conversely, assume that $\mathcal{G}_{R}(k)= \mathcal{G}_{R_1}(k)\otimes \cdots \otimes \mathcal{G}_{R_{s+1}}(k)$. We need to prove that there is at most one index $j \in I_{s+1}$ such that $-1\in U_{R_j,k}$. Suppose
		the contrary happens. Without loss of generality we may assume that, for some integer $t$ with $2\leq t \leq s+1$, we have $-1\not \in U_{R_j,k}$ for $1\le j\le t$ and $-1\in U_{R_j,k}$ for $t<j\le s+1$. 
		If $t<s+1$, then since $-1\in U_{R_{s+1},k}$, by \eqref{eq:  GR1k x GR2k} we have that 
		$$ \mathcal{G}_{R}(k) = \mathcal{G}_{R_1\times \cdots \times R_{s}}(k) \otimes \mathcal{G}_{R_{s+1}}(k).$$ 
		Similarly, if $t<s$ then $\mathcal{G}_{R_1\times \cdots \times R_{s}}(k)=\mathcal{G}_{R_1\times \cdots \times R_{s-1}}(k) \otimes \mathcal{G}_{R_s}(k)$ and, inductively, we obtain that
		$$\mathcal{G}_{R}(k)=\mathcal{G}_{R_1\times \cdots \times R_{t}}(k) \otimes\mathcal{G}_{R_{t+1}}(k)\otimes \cdots \otimes \mathcal{G}_{R_{s+1}}(k).$$
		Comparing this with the assumption $\mathcal{G}_{R}(k) = \mathcal{G}_{R_1}(k)\otimes \cdots \otimes \mathcal{G}_{R_{s+1}}(k)$ we obtain that $\mathcal{G}_{R_1\times \cdots \times R_{t}}(k)=\mathcal{G}_{R_1}(k)\otimes \cdots \otimes \mathcal{G}_{R_{t}}(k)$. Since $2\le t\le s$, by the inductive hypothesis we have that there is at most one $j$ between $1$ and $t$ such that $-1\not\in U_{R_j,k}$, but this contradicts the definition of $t$, since the assumption $t\ge 2$ implies that $t=s+1$, that is $-1\not\in U_{R_j,k}$ for $1\le j\le s+1$. 
		Since 
		$$T_{R,k} = \big(U_{R_1,k} \times\cdots\times U_{R_{s+1},k}\big) \cup -\big(U_{R_1,k} \times\cdots\times U_{R_{s+1},k} \big),$$ 
		it follows that $(-1,1,\ldots,1)\not\in T_{R,k}$. On the other hand, we have however that, 
		$(-1,1,\ldots,1) \in T_{R_1,k} \times \cdots \times T_{R_{s+1},k}$ and, consequently, 
		$$ \mathcal{G}_{R}(k) \ne X(R, T_{R_1,k} \times \cdots \times T_{R_{s+1},k}) = \mathcal{G}_{R_1}(k)\otimes \cdots\otimes \mathcal{G}_{R_{s+1}}(k). $$
		This contradiction shows that there is at most one $j \in I_{s+1}$ such that $-1\not \in U_{R_{j},k}$.
	\end{proof}
	
	We now illustrate the theorem with cyclic groups of even order.
	
	\begin{exam}
		Let $R=\mathbb{Z}_{2p^2}$ and let $k=p(p-1)$ with $p$ an odd prime. By the Euler-Fermat Theorem 
		$$ a^{p(p-1)} \equiv 1 \pmod{2p^2}$$ 
		and so 
		$$ G_{R}(p(p-1))=\vec{C}_{2p^{2}} \qquad \text{and} \qquad \mathcal{G}_{R}(p(p-1))=C_{2p^{2}}. $$ 
		It is well-known that if $n$ is odd then $C_{2n}\simeq K_{2}\otimes C_{n}$, thus in our case we get 
		$$C_{2p^{2}}\simeq K_{2}\otimes C_{p^{2}}.$$
		
		On the other hand, recall that $\Z_{2p^{2}}\simeq \Z_2 \times \Z_{p^{2}}$, in this case notice that $-1 \not\in U_{\Z_p^{2},p(p-1)}$, and $-1=1\in U_{\Z_2,p(p-1)}$. 
		Therefore, by Theorem \ref{thm: conditions for mathcalGRk Kronecker}, we obtain that 
		$$\mathcal{G}_{R}(p(p-1))\simeq \mathcal{G}_{\Z_2}(p(p-1))\otimes \mathcal{G}_{\Z_p}(p(p-1)).$$ 
		It can be checked that, $\mathcal{G}_{\Z_2}(p(p-1))=K_2$ and $\mathcal{G}_{\Z_p}(p(p-1))=C_{p^{2}}$.
		\hfill $\diamond$
	\end{exam}

	\subsection*{Reduced rings}
	When $R$ is a reduced finite ring, i.e.\@ when $R \simeq \ff_{q_1} \times \cdots \times \ff_{q_{s}}$, as a particular case of the previous result, we have that if $k\in \N$, then 
	\begin{equation} \label{eq: GRk reduced R}
		\mathcal{G}_{R}(k) \simeq  \bigotimes_{i=1}^s \mathcal{G}_{\ff_{q_i}}(k) \qquad  \Leftrightarrow \qquad 
		-1\not \in U_{\ff_{q_j},k} \quad \text{for at most one $j \in I_s$}
	\end{equation}
	(equivalently, $q_j$ is odd and $k_j \nmid \frac{q_j-1}{2}$).
	
	We next give an example to show that \eqref{eq: GRk reduced R} can fail when $-1\not \in U_{R_i,k}$ for more than one index.
	\begin{exam} 
		Let $R=\mathbb{Z}_p \times \mathbb{Z}_p$ with $p$ an odd prime and let $k=p-1$. In this case, the local components of $R$ are $R_1=R_2=\mathbb{Z}_p$. 
		Recall that the regularity degrees of $\mathcal{G}_R(k)$ and $\mathcal{G}_{R_1}(k)\otimes \mathcal{G}_{R_2}(k)$ are $|T_{R,k}|$ and $|T_{R_1,k}\times T_{R_2,k}|$, respectively.
		
		Now, notice that $-1\not \in U_{R_i,k}$ for $i=1,2$. Indeed, the unique non-zero $(p-1)$-th power of $\mathbb{Z}_p$ is $1$ due to the little Fermat's theorem. Thus, $U_{R_i,k}=\{1\}$ for $i=1,2$ which implies that $T_{R_i,k}=\{1,-1\}$ and hence 
		$$
		T_{R_1,k} \times T_{R_2,k}=\{(1,-1)(-1,1),(1,1),(-1,-1)\}. 
		$$
		On the other hand, for any $u \in R^*$, we have that  $u^{p-1}=(1,1)$, which implies that $U_{R,k}=\{(1,1)\}$ and hence
		$$
		T_{R,k} = \{(1,1),(-1,-1)\}.
		$$
		Therefore, we obtain that 
		$$
		\mathcal{G}_{R}(k) \not \simeq \mathcal{G}_{R_1}(k) 	\otimes \mathcal{G}_{R_2}(k),
		$$
		since they have different regularity degrees.
		\hfill $\diamond$
	\end{exam}

	We now give a general decomposition of $\mathcal{G}_{R}(k)$ into in terms of the symmetrized graph $\mathcal{G}_{R_{red}}(k)$ associated to the reduced ring $R_{red}$ of $R$.
	In the notations of \eqref{artin desc} and \eqref{notations R}, we have the following.
	\begin{thm} \label{thm: product decomposition for mathcal GRk}
		Let $R$ be a finite commutative ring with identity with Artin's decomposition $R \simeq R_1 \times \cdots \times R_s$ and  
		let $k\in \mathbb{N}$ be such that $(k,|R|)=1$. Then, we have that   
		\begin{equation} \label{reduced}
			\mathcal{G}_{R}(k) \simeq \big( \bigotimes_{i=1}^s \mathcal{G}_{\ff_{q_i}}(k) \big) \otimes \mathring{K}_m \simeq \mathcal{G}_{R_{red}}(k) \otimes \mathring{K}_m \simeq (\mathcal{G}_{R_{red}}(k))^{(m)},
		\end{equation}	
		where $R_{red} = \ff_{q_1} \times \cdots \times \ff_{q_s}$ and $m=|\frak{m}_1| \cdots |\frak{m}_s|$,
		if and only if there is at most one index $j \in I_s$ such that $-1 \notin U_{R_j,k}$.
	\end{thm}
	
	\begin{proof}
		By 
		Theorem \ref{thm: conditions for mathcalGRk Kronecker}, we have that 
		$\mathcal{G}_{R}(k) \simeq \mathcal{G}_{R_1}(k) \otimes \cdots \otimes \mathcal{G}_{R_s}(k)$ if and only if there is at most one 
		$j \in I_s$ such that $-1 \notin U_{R_j,k}$. 
		Now, by Theorem \ref{teo: blowup GRk}, since the $R_i$'s are local rings, by using the associativity and commutativity of the 
		$\otimes$ product and \eqref{eq: GRk reduced R} we get
		\begin{align*}
			\mathcal{G}_{R}(k) 
			\simeq  \bigotimes_{i=1}^s \big( \mathcal{G}_{\ff_{q_i}}(k) \otimes \mathring{K}_{m_i}\big)  
			\simeq  \big( \bigotimes_{i=1}^s \mathcal{G}_{\ff_{q_i}}(k) \big) \otimes \big( \bigotimes_{i=1}^s \mathring{K}_{m_i} \big) 
			\simeq  \big( \bigotimes_{i=1}^s \mathcal{G}_{\ff_{q_i}}(k) \big) \otimes \mathring{K}_m.
		\end{align*}
		Now, since $(k,|R|)=1$ is equivalent to $(k,|R_i|)=1$ for all $i\in I_s$, by \eqref{eq: GRk reduced R} we have that 
		$$ \mathcal{G}_{\ff_{q_1}}(k) \otimes \cdots \otimes \mathcal{G}_{\ff_{q_s}}(k) \simeq \mathcal{G}_{\ff_{q_1} \times \cdots \times \ff_{q_s}} (k)= \mathcal{G}_{R_{red}}(k), $$ 
		since by Corollary \ref{-1RFq} we have that $-1\in U_{R_j,k}$ if and only if $-1\in U_{\ff_{q_j},k}$.
	\end{proof}
	
	We now illustrate the theorem for a small ring.
	
	\begin{exam}
		Let $R=\Z_{4}\times \Z_4$ and let $k=3$. Notice that the nilradical of $R$ is $\mathcal{N}_{R} = \Z_2 \times \Z_2$, hence has size $4$, and that
		$$ T_{R,3}=U_{R,3}=\{(1,1),(1,-1),(-1,1), (-1,-1)\}, $$ 
		which implies that any element of $R$ has exactly $4$ neighbors. 
		
		A direct computation allows us to conclude that each of the sets 
		\begin{align*}
			&V_1=\{(0,0),(2,2),(0,2),(2,0)\}, &  & \quad V_2=\{(1,1),(1,-1),(-1,1), (-1,-1)\}, \\  
			&V_3=\{(1,2),(-1,0),(-1,2),(1,0)\}, & & \quad  V_4=\{(2,-1),(0,1),(2,1),(0,-1)\},
		\end{align*}
		are independent subsets of vertices. Moreover, if $N_S$ denotes the set of neighbours of a subset $S$ of the vertices $V(G)$ in a graph $G$, that is 
		$$
		N_S = \{w \in V(G) \, : \, vw \in E(G) \:\:  \,\text{for any $v\in S$}\}. 
		$$ 
		Then, we have  
		$$ N_{V_1}=V_2, \qquad N_{V_2}=V_1, \qquad N_{V_3}=V_4, \qquad N_{V_4}=V_3.$$
		Thus, $\mathcal{G}_{R}(3)$ is the disjoint union of two copies of the complete bipartite graph $K_{4,4}$.
		
		On the other hand, notice that since $(k,|R|)=1$ and $R_{red}=\Z_2\times \Z_2$, by Theorem \ref{thm: product decomposition for mathcal GRk} we obtain that 		
		$$ 
		\mathcal{G}_{\Z_2\times \Z_2}(3)\simeq \mathcal{G}_{\Z_2}(3) \otimes \mathcal{G}_{\Z_2}(3).
		$$
		By taking into account that $\mathcal{G}_{\Z_2}(3)= K_2$ and $K_{2}\otimes K_2$ is the disjoint union of two copies of $K_2$, we have   
		$$
		(\mathcal{G}_{R_{red}}(3))^{(4)}\simeq (K_2 \sqcup K_2)^{(4)}=K_2^{(4)}\sqcup K_2^{(4)}.
		$$
		Finally, it is well known that $K_2^{(n)}=K_{n,n}$ which implies that $K_2^{(4)}$ is the complete bipartite graph $K_{4,4}$. Therefore, we obtain that $(\mathcal{G}_{R_{red}}(3))^{(4)}$ is the disjoint union of two copies of the complete bipartite graph $K_{4,4}$ and so \eqref{reduced} holds.  
		\hfill $\diamond$
	\end{exam}

	In terms of blow-ups, by \eqref{eq: blowup} and the Kronecker product properties, the expressions in the theorem are equivalent to
	\begin{equation} \label{eq: mathcalGR blowups}
		\mathcal{G}_{R}(k) \simeq (\mathcal{G}_{\ff_{q_i}}(k))^{(m_1)} \otimes \cdots \otimes (\mathcal{G}_{\ff_{q_s}}(k))^{(m_s)} \simeq (\mathcal{G}_{R_{red}}(k))^{(m)}.
	\end{equation}
	
	\begin{rem}
		Note that Theorems \ref{teo: GRk tensor GRki's caso no local} and \ref{thm: product decomposition for mathcal GRk} are generalizations in the non-local case of Theorem \ref{teo: blowup GRk} for a local ring $R$. 
	\end{rem}
	
	As a direct consequence, we have the following corollary for rings of odd order.
	
	\begin{coro} \label{coro: general case}
		Let $R$ be a finite commutative ring with identity of odd order with Artin's decomposition $R \simeq R_1 \times \cdots \times R_s$, and let $k\in \mathbb{N}$ be such that $(k,|R|)=1$. If there is at most one index $j \in I_s$ such that $-1 \notin U_{R_j,k}$ then   
		\begin{equation} \label{reduced ki}
			\mathcal{G}_{R}(k) \simeq \big( \bigotimes_{i=1}^s \G(\tilde{k}_i,q_i) \big) ^{(m)} \simeq \big( \bigotimes_{i=1}^s \G(\tilde{k}_i,q_i) \big) \otimes \mathring{K}_m,
		\end{equation}	
		where $m=|\mathcal{N}_R|$ and for any $i=1,\ldots,s$, we have 
		$$
		\tilde{k}_i=\begin{cases}
			k_i 			& \qquad \text{if $v_2(k)< v_2(q-1)$}, \\[1mm]
			\tfrac{k_i}{2}  & \qquad \text{if $v_2(k)\ge v_2(q-1)$},
		\end{cases}
		$$
		with $k_i=(k,q_i-1)$ and $q_i$ is the size of the residue field $R_i/\frak{m}_i$.
	\end{coro}
	
	
	\begin{proof}
		It follows directly from Corollary \ref{coro: Local case mathcal G} and Theorem \ref{thm: product decomposition for mathcal GRk}. 
	\end{proof}
	
	Relative to the relation between reduced graphs and reduced rings, we have an analogous result to Theorem \ref{thm: GR_red=G_Rred} for the symmetrized graphs $\mathcal{G}_R$.
	
	\begin{thm} \label{thm: symmetrized GR_red=G_Rred}
		Let $R$ be a finite commutative ring with identity and let $k\in \mathbb{N}$ such that $(k,|R|)=1$. 
		If there is at most one index $j \in I_s$ such that $-1 \notin U_{R_j,k}$, then 
		\begin{equation} \label{eq: sym GR_red=G_Rred}
			\big(\mathcal{G}_{R}(k)\big)_{red} \simeq \mathcal{G}_{R_{red}}(k).
		\end{equation}
	\end{thm}

	\begin{proof}
		The proof is similar to the one of Theorem \ref{thm: GR_red=G_Rred}, using Theorem \ref{thm: product decomposition for mathcal GRk}.
	\end{proof}

	\goodbreak

	\section{Bipartiteness} \label{sec: bipartiteness}
	In this and the next section, we study two basic structural properties (bipartiteness and connectedness) of the graphs $G_R$ and $\mathcal{G}_R$ for an arbitrary finite commutative ring $R$, by way of Kronecker product decompositions given in Sections~{\ref{sec: product decompositions for GRk} and \ref{sec: products of symmetrized GRks}}. 
	To study the structural properties of the graphs $G_R$ we will have to distinguish whether the graphs are directed or not (the directed case being more involved).
	
	We begin by studying the bipartiteness of the graphs, so we recall first the definitions and basic properties.

	A simple graph $\G$ is \textit{bipartite} if there exist disjoint subsets $V_1,V_2$ of the set of vertices $V=V(\G)$ of $\G$ such that 
	$$ V=V_1\cup V_2 $$ 
	such that if $a,b\in V_{i}$ then $a,b$ are not neighbors in $\G$  (the sets $V_1, V_2$ are called the \textit{bipartition} of $V$).
	
	\subsubsection*{Undirected bipartite graphs} 
	We recall that an undirected  graph $\G$ is bipartite if some (and hence all) of the following equivalent conditions are satisfied:
	\begin{enumerate}[(C1).]
		\item The number $-\rho(\G)\in Spec(\G)$, where $\rho(\G)$ is the spectral radius of $\G$. \sk
		
		\item The spectrum of $\G$ is symmetric (i.e.\@ $\lambda \in Spec(\G)$ if and only $-\lambda \in Spec(\G)$). \sk
		
		\item $\G$ has only even cycles. \sk
		
		\item The chromatic number of $\G$ is 2, that is $\chi(\G)=2$. 
	\end{enumerate}
	
	Above, $Spec(\G)$ denotes the spectrum of $\G$, that is the set of all the eigenvalues $\lambda_1, \ldots, \lambda_n$ of $\G$ ($n=|V|$) counted with multiplicities. If $\G$ is undirected, the eigenvalues are real and are usually listed 
	\begin{equation} \label{eq: eigenvalues}
		\lambda_1 \ge \cdots \ge \lambda_n,
	\end{equation}
	with $\lambda_1$ referred to as the principal eigenvalue.
	The \textit{spectral radius} of $\G$ is 
	$$ 
	\rho(\G) = \max \{ |\lambda|: \lambda \in Spec(\G)\} =\lambda_1,
	$$ 
	with the convention \eqref{eq: eigenvalues}.
	It is well-known that 
	$ \rho(\G) =\Delta(\G):=\max \{\delta(v): v\in V(\G)\} $
	where $\delta(v)$ is the degree of $v$.
	When $\G$ is $m$-regular, 
	$$\rho(\G)=m$$ 
	and hence (C1) can be changed by ``$-m \in Spec(\G)$''. 
	
	\subsubsection*{Directed bipartite graphs}
	The notion of bipartiteness can be extended to directed graphs as follows. A directed graph $D$ is called bipartite if there exist a bipartition 
	$V_1, V_2$ of $V(D)$ such that  
	every arc of $D$ connects a vertex of $V_1$ to a vertex in $V_2$ or vice versa. In this case, it can be shown that the condition on $D$ to be bipartite is equivalent to  (C1) and (C2) of the above list (see Theorem 3.1 in \cite{Br}). Moreover the condition (C3)  can be changed by 
	\begin{enumerate}
		\item[(C3').] $D$ has only even directed cycles.
	\end{enumerate}

	\subsection*{Bipartiteness of blow-ups and Kronecker products of graphs}
	We will need the following facts about the bipartiteness of a blow-up and a Kronecker product of graphs in terms of the bipartiteness of the smaller graphs.
	
	We begin with the bipartiteness of a blow-up.
	\begin{lem} \label{len: blowup bipartite}
		Let $\G$ be any graph and $m\in \N$. The balanced blow-up $\G^{(m)}$ is bipartite if and only if $\G$ is bipartite.
	\end{lem}
	
	\begin{proof}
		By \eqref{eq: blowup} we have $\G^{(m)} \simeq \G \otimes \mathring{K}_m$. 
		Since the eigenvalues of the Kronecker product are the product of the eigenvalues of the factors and the eigenvalues of $\mathring{K}_m$ are $0$ and $m$, 
		we have that 
		$$ \rho(\G^{(m)}) = m\rho(\G). $$ 
		Moreover, the non-zero eigenvalues of $\G^{(m)}$ are given by $\{m\lambda\}$, where $\lambda$ is any eigenvalue of $\G$.
		Thus, if $\G$ is bipartite then $-\rho(\G)$ is an eigenvalue of $\G$ and hence $-m\rho(\G)$ is an eigenvalue of $\G^{(m)}$, which implies that $\G^{(m)}$ is bipartite. 
		
		Conversely, if $\G^{(m)}$ is bipartite then $-\rho(\G^{(m)})=-m\rho(\G)$ is an eigenvalue of $\G^{(m)}$ and hence $-\rho(\G)$ is an eigenvalue of $\G$. Therefore, the graph $\G$ is bipartite.
	\end{proof}

	Next, we recall that a Kronecker product of graphs is bipartite if and only if at least one of the factors is bipartite.
	This is a well-known result of Weichsel from 1962 (see \cite{Wei}). We give an alternative proof for completeness, using the spectral radius (Weichsel used odd cycles).
	
	\begin{lem} \label{lem: Kronecker bipartite criterion}
		The graph $\G=\G_1 \otimes \cdots \otimes \G_s$ is bipartite if and only if the graph $\G_i$ is bipartite for at least one 
		$i \in I_s$. 
	\end{lem}
	
	\begin{proof}
		Assume first that one of the graphs $\G_i$ is bipartite. 
		Without loss of generality, we can assume that $\G_1$ is bipartite. 
		Thus, there exist disjoint subsets $V_1,\, V_2$ of $V(\G_1)$, such that $V(\G_1)= V_1 \cup V_2$, which are a bipartition of $\G$.
		On the other hand, since $V(\G)=V(\G_1)\times \cdots \times V(\G_s)$, we have
		$$ 
		V(\G)= \big(V_1\times V(\G_2)\times \cdots \times V(\G_s) \big) \cup \big(V_2\times V(\G_2)\times \cdots \times V(\G_s)\big).
		$$
		By the definition of Kronecker product, $V_1\times V(\G_2)\times \cdots \times V(\G_s)$ and $V_2\times V(\G_2)\times \cdots \times V(\G_s)$ are disjoint sets forming a bipartition of $\G$, and therefore $\G$ is bipartite.
		
		Assume now that $\G$ is bipartite, so $-\rho(\G) \in Spec(\G)$  by (C1).  
		It is well-known that an eigenvalue $\lambda$ of $\G$ is a product $\lambda_{i_1} \cdots \lambda_{i_s}$ where $\lambda_{i_j}$ is some eigenvalue of $\G_j$. In this way, we clearly have 
		$$ 
		\rho(\G)=\rho(\G_1)\cdots \rho(\G_s), 
		$$ 
		which implies that there exists at least one $i\in I_s$ such that $-\rho(\G_i)\in Spec (\G_i)$, as we wanted to show. 
	\end{proof}

	\subsection*{Bipartiteness of the graphs $G_R$ and $\mathcal{G}_{R}$}
	Now, we study bipartiteness of the graphs $G_R$ and $\mathcal{G}_{R}$, by using their Kronecker product decompositions.
	
	In \cite{PV18}, we showed that all connected generalized Paley graphs $\G(k,q)$ are non-bipartite, unless the trivial case $\G(1,2)=K_2$. Since this is important in the study of connectedness and bipartiteness of the graphs $G_{R}$, so we recall this result.
	
	\begin{lem}[\cite{PV18}, Theorem 4.2] \label{lem: GP bipartito}
		Let $q=p^m$ be a fixed prime power with $m \in \N$ and $k \in \N$ with $k\mid q-1$. 
		Then, $\G(k,q)$ is non-bipartite, except when $k=2^m-1$ and $q=2^m$, in which case we have 
		$$ \G(2^m-1,2^m) \simeq K_2 \sqcup \cdots \sqcup K_2 \qquad \text{$(2^{m-1}$ copies$)$}. $$
	\end{lem}

	As a direct consequence, we obtain the following result about the bipartiteness of the graphs $G_{R}$'s and $\mathcal{G}_R$'s.
	
	\begin{prop} \label{prop: GR bipartite}
		Let $R$ be a finite commutative ring with identity with Artin's decomposition $R\simeq R_1 \times \cdots \times R_s$ 
		and let $k\in  \N$. Then, we have: 
		\begin{enumerate}[$(a)$]
			\item $G_R(k)$ is bipartite if and only if $G_{R_i}(k)$ is bipartite for some $i \in I_s$. \msk 
			
			\item If further, there is at most one index $j \in I_s$ such that $-1\not \in U_{R_j,k}$, then $\mathcal{G}_R(k)$ is bipartite if and only if  there is some index $i$, such that $\mathcal{G}_{R_i}(k)$ is bipartite. 
		\end{enumerate}
	\end{prop}

	With the two previous lemmas we are now in a position to prove that the graphs $G_R(k)$ are non-bipartite unless some of its Kronecker factors is $K_2$ or a disjoint union of $K_{2}$'s.
	
	\goodbreak 
	
	\begin{thm} \label{teo: GR bip}
		Let $R$ be a finite commutative ring with identity with Artin's decomposition $R\simeq R_1 \times \cdots \times R_s$. 
		Let $k\in \N$ be such that $(k,|R|)=1$. Then, we have 
		\begin{enumerate}[$(a)$]
			\item $G_R(k)$ is bipartite if and only if  there is some index $i \in I_s$, such that 
			\begin{equation} \label{eq: GFq K2s}
				G_{\ff_{q_i}}(k) = \G(2^{t_i}-1,2^{t_i}) \simeq K_2 \sqcup \cdots \sqcup K_2 \quad (\text{$2^{t_i-1}$ copies}).	
			\end{equation}	
			
			\item If further, there is at most one index $j \in I_s$ such that $-1\not \in U_{R_j,k}$ then, $\mathcal{G}_R(k)$ is bipartite if and only if  there is some index $i \in I_s$, such that 
			\begin{equation} \label{eq: mathcalGFq K2s}	
				\mathcal{G}_{\ff_{q_i}}(k) = \G(2^{t_i}-1,2^{t_i}) \simeq K_2 \sqcup \cdots \sqcup K_2 \quad (\text{$2^{t_i-1}$ copies}).
			\end{equation}
		\end{enumerate}
		In both items, we have that $q_i=2^{t_i}$ and $(k,q_i-1)=2^{t_i}-1$. 
	\end{thm}
	
	\begin{proof}
		($a$) By \eqref{eq: blowups for GRk} we have that 
		$$ G_R(k)=G_{R_{red}}(k)^{(m)},$$ 
		where $m=m_1 \cdots m_s$ in the notation of \eqref{notations R}.  Thus, $G_R(k)$ is bipartite if and only if $G_{R_{red}}(k)$  is bipartite, by Lemma \ref{len: blowup bipartite}.
		But 
		$$ G_{R_{red}}(k) = \G(k_1, q_1) \otimes \cdots \otimes \G(k_s, q_s)$$ 
		where $k_i = (k,q_i-1)$.
		Thus, by Lemma \ref{lem: Kronecker bipartite criterion}, we have that $G_{R_{red}}(k)$ is bipartite if and only if $\G(k_i, q_i)$ is bipartite for some index $i$. Finally, by Lemma \ref{lem: GP bipartito}, $\G(k_i, q_i)$ is bipartite if and only if 
		$\G(k_i, q_i) =\G(2^{m}-1,2^{m})$. By taking into account that $q_i=p^{t_i}$ with $p$ prime, we obtain that $m=t_i$ in this case and so we obtain \eqref{eq: GFq K2s},
		as asserted.
		
		\smallskip
		\noindent ($b$) 
		By Proposition \ref{prop: GR bipartite}, the graph $\mathcal{G}_R(k)$ is bipartite if and only if $\mathcal{G}_{R_i}(k)$ is bipartite for some $i\in I_s$. 
		Since $(k,|R|)=1$, we have $(k,|R_i|)=1$ for any $i\in I_s$, and hence $\mathcal{G}_{R_i}(k)$ is bipartite if and only if $\mathcal{G}_{\ff_{q_i}}(k)$ is bipartite, by Corollary~\ref{coro RFq conn}. 
		By Corollary~\ref{coro: Local case mathcal G}, if $q_i$ is odd then 
		$$ \mathcal{G}_{\ff_{q_i}}(k) = \G(k_i,q_i) \qquad \text{or} \qquad \mathcal{G}_{\ff_{q_i}}(k) = \G(\tfrac{k_i}{2},q_i) $$ 
		and so it can not be bipartite by Lemma~\ref{lem: GP bipartito}. Thus, if $\mathcal{G}_{\ff_{q_i}}(k)$ is bipartite then $q_i$ must be even. Since $\ff_{q_i}$ has characteristic $2$, by Theorem~\ref{teo: Gk and Gk2}, we obtain that $\mathcal{G}_{\ff_{q_i}}(k)=G_{\ff_{q_i}}(k)$ and so is a GP-graph. By Lemma~\ref{lem: GP bipartito}, we obtain that $\mathcal{G}_{\ff_{q_i}}(k)$ is bipartite if and only if 
		$\mathcal{G}_{\ff_{q_i}}(k) = \G(2^{t_i}-1,2^{t_i}) \simeq K_2 \sqcup \cdots \sqcup K_2$ 
		($2^{t_i-1}$ copies), i.e.\@ $q_i=2^{t_i}$ and $(k,q_i-1)=2^{t_i}-1$, as asserted.
	\end{proof}
	
	We now exhibit examples of local rings $R, R'$ such that $G_R(k)$ is bipartite and $G_{R'}(k)$ is not bipartite.
	
	\begin{exam}
		Let $R=\Z_{2^{\alpha}}$ and $R'=\mathbb{Z}_{5^{\beta}}$ with $\alpha, \beta \in \N$.
		In this case, the rings $R$ and $R'$ are both local rings with residue fields $\ff_2$ and $\ff_5$, respectively. 
		Thus, the graph $G_{\Z_{2^{\alpha}}}(k)$ is always a bipartite graph for any $k\in \N$, since $\G(k,2)\simeq K_2$. 
		On the other hand, by Theorem \ref{teo: GR bip}, we have that $G_{\Z_{5^{\beta}}}(k)$ is always a non-bipartite graph for any $k\in \N$, since $\G(k,5)$ is always non-bipartite by Lemma \ref{lem: GP bipartito}. 
		\hfill $\diamond$
	\end{exam}

	\section{Connectedness} \label{sec: connectivity}
	In this final section, we study the connectedness of the graphs $G_R$ and $\mathcal{G}_R$. 
	We will use the spectral radius, and for directed graphs we will also need the periods.

	We recall that a \textit{strongly connected digraph} is a directed graph in which there is a directed path in each direction between any pair of vertices of the graph. When we say connected for a general graph (pseudograph) $G$, we mean connected if $G$ is undirected and strongly connected if $G$ is directed.

	\subsection{Connectedness of Kronecker products of graphs}
	Here, we study the connectedness of the Kronecker product of connected graphs, since in general the Kronecker product of connected graphs is not necessarily connected. We will consider three cases: all the graphs are undirected, all the graphs are directed and the mixed situation.

	\subsubsection*{Product of undirected graphs}
	We begin by studying the connectedness of the Kronecker product of undirected connected graphs. Quite surprisingly, this property depends on the bipartiteness of the factors.
	
	\begin{lem}	\label{lem: G connected sii Gi bipartite}
		Let $\G_1,\ldots,\G_{s}$ be undirected graphs. 
		Then, $\G=\G_{1}\otimes \cdots \otimes \G_{s}$ is connected if and only if 
		\begin{enumerate}[$(a)$]
			\item $\G_1,\ldots,\G_{s}$ are all connected. \sk 
			
			\item $\G_{i}$ is bipartite for at most one $i \in I_s$.
		\end{enumerate}
		
	\end{lem}
	
	\begin{proof}
		$(\Rightarrow)$ Suppose first that $\G$ is connected and assume that there are  at least two different 
		indices $i,j\in I_s=\{1,\ldots,s\}$ such that $\G_i,\G_j$ are bipartite graphs. 
		Thus, we have that $-\rho(\G_i) \in Spec(\G_i)$ and $-\rho(\G_j)\in Spec(\G_j)$, and hence 
		$$\rho(\G) = \rho(\G_1) \cdots \rho(\G_i) \cdots \rho(\G_j) \cdots \rho(\G_s)= \rho(\G_1) \cdots (-\rho(\G_i))  \cdots (-\rho(\G_j)) \cdots \rho(\G_s),$$
		which implies that the multiplicity of $\rho(\G)$ is greater than $1$ and so $\G$ is disconnected.
		
		In a similar manner, if we assume that there is an index $i$ such that $\G_i$ is not connected, then the spectral radius $\rho(\G_i)$ has multiplicity greater than $1$, and therefore the multiplicity of the spectral radius $\rho(\G_1) \cdots \rho(\G_s)$ of $\G_1\otimes \cdots \otimes \G_s$ is greater than $1$, contradicting the connectedness of $\G$. Therefore, the graphs $\G_1, \ldots, \G_s$ are all connected. 
		
		\smallskip

		$(\Leftarrow)$	 Conversely, assume that $\G_1,\ldots, \G_s$  are all connected and suppose that there is at most one bipartite factor in $\G$. If there is exactly one index $j$ such that $\G_{j}$ is bipartite, then there is a unique way to write $\rho(\G)$ as $\lambda_{\G_1} \cdots \lambda_{\G_s}$ with 
		$\lambda_{\G_i} \in Spec(\G_i)$ for $i\in I_s$, and hence $\G$ is connected.
		Indeed, since $\rho(\G)=\rho(\G_1)\cdots \rho(\G_s)$, 
		if 
		$$ \rho(\G) = \lambda_{\G_1} \cdots \lambda_{\G_s} $$ 
		with $\lambda_{\G_i} \in Spec(\G_i)$ for $i\in I_s$, 
		then we have that $|\lambda_{\G_i}| = \rho(\G_i)$. 
		Since $\G_i$ is undirected, its adjacency matrix is symmetric and so all of its eigenvalues are real numbers, which implies that 
		$$ \lambda_{\G_i} = \pm \rho(\G_i).$$ 
		Since there is exactly one index $j$ such that $\G_j$ is bipartite, that is $-\rho(\G_j) \in Spec(\G_j)$ 
		and $\rho(\G)=\lambda_{\G_1} \cdots \lambda_{\G_s}$, then $\lambda_{\G_i} = \rho(\G_i)$ for all $i\in I_s$ and therefore $\G$ is connected.
		
		Finally, if $\G_i$ is non-bipartite for all $i\in I_s$. 
		A similar argument as above shows that $\G$ is connected. 
		In fact, there is a unique way to write $\rho(\G)$ as $\lambda_{\G_1} \cdots \lambda_{\G_s}$ with $\lambda_{\G_i} \in Spec(\G_i)$ for $i\in I_s$,
		namely $\lambda_{\G_i}=\rho(\G_i)$ since this is the unique eigenvalue of $\G_i$ with norm $\rho(\G_i)$.
	\end{proof}

	For the directed case let us recall the following. 
	Given $D$ a digraph, the \textit{period} or \textit{index of imprimitivity} of $D$, denoted $d=d(D)$, is the greatest common divisor between all the lengths of the directed cycles in $D$,  that is 
	\begin{equation} \label{eq: period of G}
		d=d(G)= \gcd\{ \ell(\vec{C}) : \vec{C} \text{ is a directed cycle in $D$}\}. 
	\end{equation}
	
	It is well-known that the spectrum of $D$, as a set of points in $\C$, is invariant under a rotation $\rho_d$ about the origin by the angle $\frac{2\pi}d$ (see Theorem 2.1 in \cite{Br}), i.e. 
	$$ e^{\frac{2\pi i}d} Spec(D) = \rho_d(Spec(D)) \subseteq Spec(D). $$
	Moreover, the index of imprimitivity is exactly the number of eigenvalues with absolute value equal to its spectral radius (see \cite{AGH} and \cite{HJ}). 
	Notice that for instance, $d$ is even if and only if $D$ is bipartite.

	\subsubsection*{Product of directed graphs}
	Now, we study the connectedness of the Kronecker product of directed connected graphs. Quite surprisingly, it depends on the coprimeness of all the periods of the factors.

	\begin{lem} \label{lem: comd strcon kron}
		Let $D_1,\ldots, D_s$ be directed graphs. 
		Then, 
		$D=D_1\otimes \cdots \otimes D_s$ is strongly connected if and only if 
		\begin{enumerate}[$(a)$]
			\item $D_1,\ldots, D_s$ are all strongly connected. 
			\sk 
			
			\item $\gcd(d(D_j),d(D_{j'}))=1$, for every $0\le j < j' \le s$.
		\end{enumerate}
		
	\end{lem}
	
	
	\begin{proof}
		For simplicity, we will write $d_j$ for the period $d(D_j)$ of $D_j$.
		
		\noindent ($\Rightarrow$) 
		Assume first that $(d_j,d_{j'})=\ell >1$ for some $j,j'\in I_s$ with $j \ne j'$.
		Notice that, if the spectra of a digraph as a set of points in the complex plane is invariant under a rotation about the origin by the angle $\frac{2\pi}d$, then it is invariant under a rotation by the angle $\frac{2\pi}{f}$ for all $f\mid d$.
		In particular, the spectra of $D_j$ and $D_{j'}$ are invariant under a rotation about the
		origin by the angle $\frac{2\pi}{\ell}$ and so we have that
		$$ \lambda_{\rho_\ell,j} := e^{\frac{2\pi i}{\ell}} \rho(D_j) \in Spec (D_j) 
		\qquad \text{and} \qquad 
		\lambda_{\rho_{\ell}^{-1},j'} := e^{\frac{2\pi i (\ell-1)}{\ell}} \rho(D_{j'}) \in Spec(D_{j'}).$$
		This implies that  
		$$ \lambda_{\rho_\ell,j} \lambda_{\rho_{\ell}^{-1},j'} = \rho(D_j)\rho(D_{j'}).$$ 
		In this way, we see that, 
		$$ \rho(D) = \rho(D_1)\cdots \rho(D_s) $$ 
		has multiplicity $>1$ and, therefore, the directed graph $D$ is not strongly connected.
		
		If we assume that there is an index $i \in I_s$ such that $D_i$ is not strongly connected, then the spectral radius $\rho(D_i)$ has multiplicity greater than $1$, and therefore the multiplicity of the spectral radius $\rho(D_1) \cdots \rho(D_s)$ of $D_1\otimes \cdots \otimes D_s$ is greater than $1$, contradicting the strongly connectedness of $D$. Therefore, the digraphs $D_1, \ldots, D_s$ are all strongly connected.	
		
		\smallskip
		
		\noindent ($\Leftarrow$)
		Conversely, suppose that $D_1,\ldots, D_s$ are all strongly connected, and assume that $(d_j,d_{j'})=1$ for all $j,j'\in I_s$ with $i\neq j$. 
		
		\noindent
		\textit{Claim:} Under the above hypothesis, if $\frac{\ell_1}{d_1}+\cdots+\frac{\ell_s}{d_s}\in \mathbb{Z}$ with $(\ell_1,\ldots,\ell_s)\in I_{d_1}\times \cdots \times I_{d_s}$, then
		$(\ell_1,\ldots,\ell_s)=(d_1,\ldots,d_s)$.
		
		Let $\ell_j \in I_{d_j}$ for $j \in I_s$, such that $\frac{\ell_1}{d_1}+\cdots+\frac{\ell_s}{d_s}\in \mathbb{Z}$.
		If $h=d_1\cdots d_s$, notice that 
		$$ \tfrac{\ell_1}{d_1}+\cdots+\tfrac{\ell_s}{d_s}\in \mathbb{Z} \qquad \Leftrightarrow \qquad 
		h \mid \ell_1 \tfrac{h}{d_1} + \cdots + \ell_s \tfrac{h}{d_s}. $$
		So, using that $d_j \mid h$ for $j\in I_{s}$, we have that
		$$ d_j \mid \ell_1 \tfrac{h}{d_1} + \cdots + \ell_s \tfrac{h}{d_s}. $$ 
		Hence, since $d_j \mid \frac{h}{d_{j'}}$ for all $j \ne j'$, we arrive to $d_j \mid \ell_j \frac{h}{d_j}$.
		By hypothesis $(d_j,d_{j'})=1$ for all $j \ne j'$, this implies that $(d_j,\frac{h}{d_j})=1$ and so $d_j \mid \ell_j$. Therefore, 
		since $\ell_j \le d_j$ we obtain that $\ell_j=d_j$, as claimed. \hfill $\diamond$
		
		Now, it is enough to show that the multiplicity of the eigenvalue $\rho(D)$ is $1$.
		By taking into account that all the eigenvalues of $D$ are product of the eigenvalues of its factors $D_i$, we have that there exist $\lambda_{j} \in Spec(D_i)$ for $j\in I_s$ such that 
		\begin{equation} \label{eq: ro lambda}
			\rho(D) = \lambda_{1} \cdots \lambda_{s}. 
		\end{equation}
		On the other hand, since $\rho(D)=\rho(D_1)\cdots \rho(D_s)$, we arrive to $|\lambda_{j}|=\rho(D_j)$ and so, 
		for each $j \in I_s$ there exist $\ell_{j} \in \{1,2,\ldots,d_j\}$, such that
		$$ \lambda_{j} = e^{2\pi i\frac{\ell_j}{d_j}} \rho(D_j). $$ 
		Hence, the equality \eqref{eq: ro lambda} implies that 
		$$1=e^{2\pi i\frac{\ell_1}{d_1}}\cdots e^{2\pi i\frac{\ell_s}{d_s}}= e^{2\pi i (\frac{\ell_1}{d_1}+\cdots+\frac{\ell_s}{d_s})}.$$
		Thus, $\frac{\ell_1}{d_1}+\cdots+\frac{\ell_s}{d_s}\in \mathbb{Z}$ and, by the above claim, we obtain that $\ell_{j}=d_j$ for all $j\in I_s$. So, we have that 
		$\lambda_{j}=\rho(D_j)$ for all $j\in I_s$. Therefore, the directed graph $D$ is strongly connected, as asserted.
	\end{proof}

	\subsubsection*{Product of directed graphs}
	Finally, we study the connectedness of the Kronecker product between an undirected and a directed connected graphs.
	\begin{lem} \label{lem: D vs U prod}
		Let $\G$ be an undirected graph and let $D$ be a digraph.
		Then, the graph $\G \otimes D$ is directed and we have that $ \G \otimes D$ is strongly connected if and only if
		\begin{enumerate}[$(a)$]
			\item $\G$ is connected and $D$ is strongly connected. 
			\sk 
			
			\item $\G$ or $D$ are non-bipartite.
		\end{enumerate}
	\end{lem}
	
	\begin{proof}
		We know that $\G \otimes D$ is directed. Now, if $\G$ and $D$ are bipartite or, else, if some of it is disconnected, then in the same way as in the proof of 
		Lemma \ref{lem: G connected sii Gi bipartite}, 
		we can show that $\G\otimes D$ is not connected.
		
		For the converse, assume first that $\G$ is non-bipartite. 
		If $\rho(\G)=\lambda\cdot\mu$ with $\lambda\in Spec(\G)$ and $\mu \in Spec(D)$, then
		$|\lambda|=\rho(\G)$ and $|\mu|= \rho(D)$ since 
		$$ \rho(\G\otimes D) = \rho(\G)\rho(D).$$ 
		By taking into account that $\G$ is undirected  (hence real spectrum) and non-bipartite, the only eigenvalue with absolute value equal to $\rho(\G)$
		is $\rho(\G)$ and hence $\lambda=\rho(\G)$, this implies that $\mu\in \mathbb{R}$.
		Finally, since $\rho(\G\otimes D)$ and $\rho(\G)$ are both positive, this implies that $\mu>0$ as well and so $\mu =\rho(D)$.
		Therefore $ \rho(\G\otimes D) $ 
		has multiplicity $1$ (since $\rho(\G)$ and $\rho(D)$ both have multiplicity $1$) 
		and thus $\G\otimes D$ is connected.
		
		A similar argument shows that if $D$ is non-bipartite then $\G\otimes D$ is connected.  
	\end{proof}

	\subsection{Connectedness of the graphs $G_R$ and $\mathcal{G}_{R}$}
	Now we deal with the connectedness of the graphs $G_R$ and $\mathcal{G}_{R}$.
	
	First, we compare the connectedness of $G_R(k)$ with the corresponding one of each $G_{R_i}(k)$ 
	or with $G_{R_{red}}(k)$.


	\begin{prop} \label{prop: GRk connected iff GRki connected}
		Let $R$ be a finite commutative ring with identity with Artin's decomposition $R=R_1\times \cdots \times R_s$. 
		Then, for any $k \in \N$ we have:
		\begin{enumerate}[$(a)$]
			\item $G_R(k)$ is undirected if and only if $G_{R_i}(k)$ is undirected for every $i \in I_s$. \sk 
			
			\item If $G_R(k)$ is  (strongly) connected then $G_{R_i}(k)$ is (strongly) connected for every $i \in I_s$.  
		\end{enumerate}
		If further $(k,|R|)=1$, the following holds: 
		\begin{enumerate}[$(a)$]
			\setcounter{enumi}{2}
			
			\item $G_R(k)$ is (strongly) connected if and only if $G_{R_{red}}(k)$ is (strongly) connected. 
		\end{enumerate}
	\end{prop}
	
	\begin{proof}
		($a$)  This follows immediately from the fact that $(-1,\ldots,-1)\in U_{R}(k)$ 
		if and only if $-1\in U_{R_i}(k)$ for all $i \in I_s$ (see Lemma \ref{lem: GRk undirected conditions}). \sk 
		
		\noindent ($b$) Assume that $G_{R}(k)$ is connected and let $a,b\in R_{i}$. Since $G_{R}(k)$ is connected, then there exist a walk from 
		$$ (0,\ldots,0,a,0,\ldots,0) \quad \text{to} \quad (0,\ldots,0,b,0,\ldots,0),$$ 
		where in all the elements of the walk, all of the coordinates different from $i$ are zero. 
		By definition of Kronecker product, this walk induces a walk from $a$ to $b$ in $G_{R_i}(k)$ and hence $G_{R_i}$ is connected.
		Since $i$ is arbitrary, $G_{R_i}(k)$ is connected for every $i \in I_s$. \sk 
		
		\noindent ($c$) The hypothesis $(k,|R|)=1$ ensures that we can apply Theorem \ref{teo: GRk tensor GRki's caso no local}. By this theorem, we have that $G_{R}(k) \simeq G_{R_{red}}(k)^{(m)}$ where $m$ is the size of the nilradical of $R$.
		It is a classic result that the blow-up of connected graphs is connected (see \cite{Ha}), 
		so if $G_{R_{red}}(k)$ is connected then $G_{R}$ is connected.
		
		Conversely if $G_{R}$ is connected, then its principal eigenvalue (i.e.\@ the regularity degree) $\lambda_R$ has multiplicity $1$. 
		By Theorem \ref{teo: GRk tensor GRki's caso no local} we have that $G_{R}(k)\simeq G_{R_{red}}(k)\otimes \mathring{K}_{m}$. 
		It is well-known that $\mathring{K}_{m}$ has only eigenvalues $0$ and $m$ with multiplicity $m$ and $1$, respectively. 
		Then, by recalling that the eigenvalues of the Kronecker product are the product of the eigenvalues of its factors, 
		we have that 
		$\lambda_R = m \lambda_{R_{red}}$, 
		where $\lambda_{R_{red}}$ is the principal eigenvalue of $G_{R_{red}}(k)$, and so $\lambda_{R_{red}}$ has multiplicity $1$ as well. 
		Hence, $G_{R_{red}}(k)$ is also connected, as asserted.
	\end{proof}
	
	We now study the connectedness of the Kronecker product of GP-graphs $\G(k_i,q_i)$ under certain conditions.
	
	\begin{prop} \label{prop: conn ff kron prod}
		Let $\ff_{q_1},\ldots, \ff_{q_s}$ be finite fields and 
		let $k\in \N$.
		Consider the GP-graphs $\G_i=\G(k_i,q_i)$ for $i\in I_s$  where $k_i=(k,q_i-1)$. 
	
	\begin{enumerate}[$(a)$]
		\item If $\G_i$ is undirected for every $i\in I_s$, then $\G_1 \otimes \cdots \otimes \G_s$ is connected if and only if \sk
		
		\begin{enumerate}[$\circ$]	
			\item  $\frac{q_i-1}{k_i}\dagger q_i-1$ for any $i\in I_s$. \sk
			
			\item  $\G_j \simeq \G(1,2)$ for at most one $j \in I_s$.
		\end{enumerate}
		
		\msk
		
		\item If $\G_i$ is directed for all $i\in I_s$, then $\G_1 \otimes \cdots \otimes \G_s$ is strongly connected if and only if \sk
		
		\begin{enumerate}[$\circ$]
			\item  $\frac{q_i-1}{k_i}\dagger q_i-1$ for any $i\in I_s$. \sk 
			
			\item There are no two $j,j' \in I_s$ such that $\G_j \simeq \G(p-1,p) \simeq \G_{j'}$ with $p$ an odd prime.
		\end{enumerate}
	\end{enumerate}
\end{prop}

\begin{proof}
	Recall that the condition $\frac{q_i-1}{k_i}\dagger q_i-1$ is equivalent to state that $\G(k_i,q_i)$ is (strongly) connected for all $i \in I_s$. 
	Let 
	$$ \G = \G_1 \otimes \cdots \otimes \G_s . $$ 
	
	\noindent ($a$) 
	Assume first that $\G_i$ are undirected for all $i\in I_s$. 
	By Lemma \ref{lem: G connected sii Gi bipartite}, the graph $\G$ is connected if and only if there is at most one index $j \in I_{s}$ such that $\G_j$ is bipartite.
	By Lemma \ref{lem: GP bipartito}, the unique bipartite generalized Paley graph is $\G(1,2)$, since there is no two indices $j,j'$ such that $\G_j \simeq \G(1,2) \simeq \G_{j'}$, we obtain that $\G$ is connected.

	\noindent ($b$) 
	Now, assume that $\G_i$ are directed for all $i\in I_s$. 
	Since there is no two indices $j,j'$ such that $\G_j \simeq \G(p-1,p) \simeq \G_{j'}$ for some prime $p$,
	Proposition~7.2 from \cite{PV19} implies that all of the periods $d_i$ of 
	$\G_i$'s with $i \in {I}_{s}$ are pairwise coprime, that is $(d_j,d_{j'})=1$ for all $j,j' \in {I}_s$. 
	Therefore, Lemma \ref{lem: comd strcon kron} 
	implies that  $\G$ is strongly connected, as asserted.	
\end{proof}

As a direct consequence of previous lemmas and the above proposition, we obtain the following result giving necessary and sufficient conditions for the connectedness of $G_R(k)$ for $R$ a finite commutative ring.

\begin{thm} \label{teo: GR connected}
	Let $R$ be a finite commutative ring with identity with Artin's decomposition $R=R_1\times \cdots \times R_s$ and 
	let $k\in \N$. If $(k,|R|)=1$, then $G_{R}(k)$ is (strongly) connected if and only if 
	\begin{enumerate}[$(a)$]
		\item $\frac{q_i-1}{k_i}\dagger q_i-1$, where $k_i = (k,q_i-1)$ for all $i \in I_s$. \sk
		
		\item There are no two indices $i,j\in I_s$ such that 
		$$ G_{R_{i}/\frak m_i}(k) \simeq \G(p-1,p) \simeq G_{R_{j}/\frak m_j}(k) $$ 
		for some prime $p$.
	\end{enumerate}
\end{thm}

\begin{proof}
	\noindent $(\Leftarrow)$ Suppose that there are no two indices $i,j$ such that $\G(k_i,q_i)\simeq \G(p-1,p) \simeq \G(k_j,q_j)$ for some prime $p$ and assume that $\frac{q_i-1}{k_i} \dagger q_i-1$ where $k_i = (k,q_i-1)$ for all $i \in I_s$.
	
	By item ($c$) of Proposition \ref{prop: GRk connected iff GRki connected}, 
	it is enough to see that $G_{R_{red}}(k)$ is connected (strongly connected).
	
	Notice that we can separate the indices in $I_s$ into the following disjoint subsets:
	\begin{align*}
		&  \vec{I}_{s} = \{i\in I_{s}: \, k_i\nmid \tfrac{q_i-1}{2} \, \text{with $q_i$ odd}\}, \\[1mm]
		&  \bar{I}_{s} = \{j\in I_{s}: \, k_j\mid \tfrac{q_j-1}{2}\}\cup \{j\in I_{s}: \, \text{$q_j$ is even}\},
	\end{align*}	
	and thus we can consider the following graphs
	$$ \G_{D}(k) = \bigotimes_{i\in \vec{I}_{s}} \G(k_i,q_i) \qquad \text{and} \qquad \G_{U}(k) = \bigotimes_{j\in \bar{I}_{s}} \G(k_j,q_j).$$
	
	By item ($a$) of Proposition \ref{prop: GRk connected iff GRki connected}, we have that $\G_{D}$ is directed and $\G_{U}$ is undirected (when the corresponding set of indices is non-empty), moreover we have that
	$$G_{R_{red}}(k)= \G_{D}(k) \otimes \G_{U}(k).$$
	By Proposition \ref{prop: conn ff kron prod} we obtain that $\G_{U}(k)$ is connected and $\G_{D}(k)$ is strongly connected.
	
	Thus, by Lemmas \ref{lem: Kronecker bipartite criterion} and \ref{lem: GP bipartito}, the digraph $\G_{D}(k)$ is always non-bipartite.
	Hence, Lemma~\ref{lem: D vs U prod} implies that $G_{R_{red}}(k)$ is connected (strongly connected).
	
	\smallskip
	
	\noindent $(\Rightarrow)$ 
	If $G_R(k)$ is (strongly) connected, then $G_{R_{red}}(k)$ is (strongly) connected and, since $G_{R_{red}}(k) \simeq \G(k_1,q_1) \otimes \cdots \otimes \G(k_s,q_s)$, then $\frac{q_i-1}{k_i}\dagger q_i-1$ for all $i\in I_s$ by Proposition \ref{prop: conn ff kron prod}.

	On the other hand, if there exist $i,j$ such that 
	$\G(k_i,q_i) \simeq \G(p-1,p) \simeq \G(k_j,q_j)$,  
	then $\G(k_i,q_i), \G(k_j,q_j)$ are both bipartite when $p=2$ or else both have the same index of imprimitivity $p$. 
	In any case, we have that the Kronecker product 
	$$ \G(k_i,q_i)\otimes \G(k_j,q_j)$$ 
	is not connected (not strongly connected) by Lemmas \ref{lem: G connected sii Gi bipartite} and \ref{lem: comd strcon kron}. 
	Therefore, we have that $G_{R_{red}}(k)$ is not connected and hence $G_{R}(k)$ is not connected, as it was to be shown.
\end{proof}

\begin{rem}
	For any prime $p>2$ the graph $\G(p-1,p)$ is the directed $p$-cycle $\vec{C}_{p}$.
\end{rem}

As a direct consequence, we obtain the following theorem that states the equivalence between the connectedness of $G_R$ and $\mathcal{G}_{R}$.

\begin{thm} \label{teo: GR sim conn}
	Let $R$ be a finite commutative ring with identity with Artin's decomposition $R=R_1\times \cdots \times R_s$. Let $k\in \N$ with $(k,|R|)=1$ such that $\frac{q_i-1}{k_i}\dagger q_i-1$ where $k_i = (k,q_i-1)$ for all $i \in I_s$.
	If there is at most one $j \in I_s$ such that $-1 \notin U_{R_j,k}$ then 
	$$ \mathcal{G}_{R}(k) \text{ is connected} \qquad \Leftrightarrow \qquad G_{R}(k) \text{ is (strongly) connected}. $$
\end{thm}

\begin{proof}
	Notice that if $-1 \in U_{R_i,k}$ for any $i \in I_s$, then $-1\in U_{R,k}$ and so $U_{R,k}=-U_{R,k}$ which implies that $U_{R,k}=T_{R,k}$. Therefore, we obtain that $G_{R}(k)=\mathcal{G}_{R}(k)$, which clearly implies the statement of the theorem. 
	
	Thus, let us assume that there exist one index such that $-1\not \in U_{R_j,k}$. Without loss of generality we can assume that $-1\not \in U_{R_s,k}$. 
	Hence, we have that $-1 \in U_{R_i,k}$ for any $i=1,\ldots, s-1$, which implies that $U_{R_1\times \cdots \times R_{s-1},k}=T_{R_1\times \cdots \times R_{s-1},k}$ and therefore
	$$
	\mathcal{G}_{R_1\times \cdots \times R_{s-1}}(k)=G_{R_1\times \cdots \times R_{s-1}}(k).
	$$
	On the other hand, since $(k,|R|)=1$, then we have that $(k,|R_s|)=1$. Thus, by ($a$) in Corollary~\ref{-1RFq} we have that 
	$-1 \not \in U_{R_s}$ if and only if $-1 \not \in U_{\ff_{q_s}}$, where $\ff_{q_s} \simeq R_{s}/ \frak{m}_s$. Therefore,  $q_s$ must be necessarily odd in this case. Since $R_s$ is of odd order, by Corollary~\ref{coro: Local case mathcal G} we have that
	$$
	\mathcal{G}_{R_s}(k)=\G(\tfrac{k_s}{2},q_s)^{(m_s)},
	$$
	which implies that $\mathcal{G}_{R_s}(k)$ cannot be a bipartite graph by Corollary \ref{coro RFq conn} and Lemma \ref{lem: GP bipartito}.
	
	\smallskip	
	
	$\bullet$ If $G_{R}(k)$ is connected, then $G_{R_1\times \cdots \times R_{s-1}}(k)$ and $G_{R_s}(k)$ are connected, which implies that $\mathcal{G}_{R_1\times \cdots \times R_{s-1}}(k)$ and $\G(k_s,q_s)$ are connected. Since $\G(k_s,q_s)$ is a subgraph of $\G(\frac{k_s}{2},q_s)$ then $\G(\frac{k_s}{2},q_s)$ is connected as well. Since $R_s$ has odd characteristic and $(k,|R_s|)=1$, by Corollary \ref{coro: Local case mathcal G} we have 
	$\mathcal{G}_{R_s}(k) \simeq \G(\tfrac{k_s}{2},q_s)^{(m_s)}$, 
	which implies that $\mathcal{G}_{R_s}(k)$ is connected.  By the above paragraph, the graph $\mathcal{G}_{R_s}$ is non-bipartite and so by Lemma \ref{lem: G connected sii Gi bipartite} we obtain that 
	$$ \mathcal{G}_{R_1 \times \cdots \times R_{s-1}}(k) \otimes \mathcal{G}_{R_s}(k) $$ 
	is connected and therefore $\mathcal{G}_{R}(k)$ is connected, as desired.
	
	\smallskip
	
	$\bullet$ If $\mathcal{G}_{R}(k)$ is connected then $\mathcal{G}_{R_1\times \cdots \times R_{s-1}}(k)$ and $\mathcal{G}_{R_s}(k)$ are connected. 
	Thus, we obtain that $G_{R_1\times \cdots \times R_{s-1}}$ is an undirected connected graph, then $G_{R_i}$ is connected and undirected for any $i=1,\ldots,s-1$, and hence $\tfrac{q_i-1}{k_i}\dagger q_i-1$ for any $i=1,\ldots,s-1$. Then, for any $q_i$ odd we have $G_{\ff_{q_i}}(k) \simeq \G(q_i-1,q_i)$, since $\G(q_i-1,q_i)$ is directed. On the other hand, there is at most one index $i\in \{1,\ldots,s-1\}$ such that $G_{\ff_{q_i}}(k)\simeq \G(1,2)$ for $i\in \{1,\ldots,s-1\}$ by Lemma \ref{prop: GR bipartite}.
	Since $\mathcal{G}_{R_s}(k)$ is connected and $-1\not \in U_{R_s,k}$ then $\G(\tfrac{k_s}{2},q_s)$ is connected. By taking into account that (see Theorem 3.5 from \cite{PV19})
	$$
	\G(\tfrac{k_s}{2},q_s) = \overset{_{\rightarrow}}{\G}(k_s,q_s) \cup \overset{_{\leftarrow}}{\G}(k_s,q_s) 
	$$ 
	we obtain that the graph $\G(k_s,q_s)$ is a strongly connected digraph, which implies that $\tfrac{q_s-1}{k_s}\dagger q_s-1$.
	Moreover, since $q_s$ is odd we have that $G_{R_s}(k)$ is directed and so cannot be equal to $\G(1,2)$. Therefore, 
	$G_{R_1 \times \cdots \times R_{s}}(k)$ is connected by Theorem \ref{teo: GR connected}.
\end{proof}

As a direct consequence, we obtain the following characterization of the connectedness of $\mathcal{G}_{R}(k)$.

\begin{coro} \label{coro: GR sim conn}
	Let $R$ be a finite commutative ring with identity with Artin's decomposition $R=R_1\times \cdots \times R_s$ 
	and let $k\in \N$ with $(k,|R|)=1$. If there is at most one  $j \in I_s$ such that $-1 \notin U_{R_j,k}$ then $\mathcal{G}_{R}(k)$ is connected if and only if 
	\begin{enumerate}[$(a)$]
		\item $\frac{q_i-1}{k_i} \dagger q_i-1$, where $k_i = (k,q_i-1)$ for all $i \in I_s$. \sk 
		
		\item There are no two indices $i,j\in I_s$ such that $\mathcal{G}_{R_{i}/\frak m_i}(k)\simeq \G(1,2) \simeq \mathcal{G}_{R_{j}/\frak m_j}(k)$.
	\end{enumerate}
\end{coro}

\begin{proof}
	By Theorem \ref{thm: conditions for mathcalGRk Kronecker} we have that 
	$$
	\mathcal{G}_{R}(k)=\mathcal{G}_{R_1}(k)\otimes \cdots \otimes \mathcal{G}_{R_s}(k).
	$$
	If there are two indices such that $\mathcal{G}_{R_{i}/\frak m_i}(k)\simeq \G(1,2) \simeq \mathcal{G}_{R_{j}/\frak m_j}(k)$, then $\mathcal{G}_{R_{i}}(k)$ and $\mathcal{G}_{R_{j}}(k)$ are bipartite and hence $\mathcal{G}_{R}(k)$ is disconnected by Lemma \ref{lem: G connected sii Gi bipartite}.
	
	Now, assume that there are no two indices $i,j\in I_s$ such that $\mathcal{G}_{R_{i}/\frak m_i}(k)\simeq \G(1,2) \simeq \mathcal{G}_{R_{j}/\frak m_j}(k)$. 
	If $-1\in U_{R_i,k}$ for any $i\in I_s$, then $\mathcal{G}_{R_{i}}(k)=G_{R_i}(k)$ for any $i\in I_s$, in particular the graphs $G_{R_i}(k)$'s are all undirected, and so the graphs $G_{R_i/\frak{m}_i}(k)$'s cannot be isomorphic to $\G(p-1,p)$ with $p$ odd. By taking into account that 
	$$ \mathcal{G}_{R_{i}/\frak m_i}(k) = G_{R_{i}/\frak m_i}(k)$$ 
	for any $i\in I_s$, so there are no two indices $i,j\in I_s$ such that $G_{R_{i}/\frak m_i}(k)\simeq \G(1,2) \simeq G_{R_{j}/\frak m_j}(k)$ which implies that $G_{R}(k)$ is connected by Theorem \ref{teo: GR connected}. Therefore, by Theorem~\ref{teo: GR sim conn} we obtain that $\mathcal{G}_{R}(k)$ is connected, as asserted.
	
	A similar argument allow us to obtain that $\mathcal{G}_{R}(k)$ is connected, if there is exactly one index $i$ such that $-1\not\in U_{R_i,k}$.
\end{proof}

We finish the paper with an application of the last results.

\begin{exam}
	Let $R$ be a finite commutative ring with $|R|$ odd and Artin's decomposition $R=R_1\times \cdots \times R_s$.
	If $k\in \N$ with $(k,|R|)=1$ such that $-1\in U_{R_i}$ for any $i\in I_s$, then $G_{R}(k)=\mathcal{G}_{R}(k)$. 
	Is this graph connected?
	Since $-1\in U_{R_i,k}$ we have that $-1\in U_{R_i/\frak{m}_i,k}$, by Lemma \ref{lem: condition for -1 in URk, R local}. 
	Now, by taking into account that $U_{\ff_q,q-1}=\{1\}$, we obtain that 
	$$ G_{R_i/\frak{m}_i}(k) \not \simeq \G(q_i-1,q_i) $$ 
	for all $i\in I_s$. Therefore, the graph $G_{R}(k) = \mathcal{G}_{R}(k)$ is connected by Theorem \ref{teo: GR sim conn}.
	\hfill $\diamond$
\end{exam}


\section*{Final remarks}
We have studied basic structural properties of the graphs $G_R(k)=Cay(R, U_k)$ and $\mathcal{G}_R(k)=Cay(R, U_k \cup -U_k)$, with  $U_k=\{x^k : x \in R^*\}$ where $R^*$ is the group of units of $R$, for $R$ a finite commutative ring with identity. 
In particular, we gave necessary and sufficient conditions for directedness, bipartiteness and connectedness of these graphs. Also, we have given some decompositions for these graphs: blow-up decompositions when $R$ is local, and Kronecker product decompositions in the general case (i.e., $R$ not necessarily local). Moreover, we have a related reduction of the graphs with the reduction of the rings.

In a forthcoming work, we will provide the connected components of these graphs, in the case when they are not connected, and we will study spectral properties such as the spectrum and energy (as well as isospectral and equienergetic problems).

Another interesting question is to study the primeness of these graphs (see Remark \ref{rem: primeness}) 
The problem of characterizing when $G_R(k)$ is prime for any finite commutative ring for a general integer $k$ is still open, and our characterization of connected $G_R(k)$ is a first step in this direction. The primeness of $\mathcal{G}_R(k)$ will probably follow from the one of the $G_R$'s, due to the Kronecker product decompositions and Theorem \ref{teo: Gk and Gk2}.

\goodbreak
\subsection*{Acknowledgements}
The first author wishes to thank the Guangdong Technion Israel Institute of Technology (GTIIT) in Shantou, China, for its hospitality, facilities, and friendly atmosphere, during his academic visit on winter 2025, where this work was almost finished.

\end{document}